\documentclass[11pt]{article}
\usepackage{a4wide}

\usepackage{amsmath,amsfonts,amssymb,amsthm,amscd}
\usepackage{epsfig}
\usepackage{hyperref}
\usepackage{breakurl}

\usepackage{array}
\usepackage{color}

\usepackage{enumerate} 

\definecolor{modra3}{rgb}{.1,.0,.4}
\hypersetup{colorlinks=true, linkcolor=modra3,
urlcolor=modra3, citecolor=modra3}

\newtheorem{theorem}{Theorem}[section]
\newtheorem{lemma}[theorem]{Lemma}
\newtheorem{observation}[theorem]{Observation}
\newtheorem{corollary}[theorem]{Corollary}
\newtheorem{proposition}[theorem]{Proposition}
\newtheorem{claim}[theorem]{Claim}
\newtheorem{conjecture}{Conjecture}

\newtheorem*{claim_}{Claim}

\def\mycite{~\cite}
\def\mydotl{}
\def\mycommal{}
\def\mydotr{.  }
\def\mycommar{,  }

\def\hell{l}

\newcolumntype{L}[1]{>{\raggedright\let\newline\\\arraybackslash\hspace{0pt}}m{#1}}
\newcolumntype{C}[1]{>{\centering\let\newline\\\arraybackslash\hspace{0pt}}m{#1}}
\newcolumntype{R}[1]{>{\raggedleft\let\newline\\\arraybackslash\hspace{0pt}}m{#1}}

\begin{document}

\title{Crossing numbers and combinatorial characterization of monotone drawings of $K_n$\thanks{
The authors were supported by the grant GA\v{C}R GIG/11/E023 GraDR in the framework of ESF EUROGIGA program. The first and the third author were also supported by the Grant Agency of the Charles University, GAUK 1262213, and by the grant SVV-2013-267313 (Discrete Models and Algorithms). The third author was also partially supported by ERC Advanced Research Grant no 267165 (DISCONV). The second author gratefully acknowledges support from the Swiss National Science Foundation Grant PBELP2\_146705.
} %end of thanks
} %end of title

\author{Martin Balko\thanks{
Department of Applied Mathematics, 
Charles University, Faculty of Mathematics and Physics, 
Malostransk\'e n\'am.~25, 118 00~ Praha 1, Czech Republic; 
\texttt{balko@kam.mff.cuni.cz}
}
\and Radoslav Fulek\thanks{Department of Applied Mathematics, 
Charles University, Faculty of Mathematics and Physics, 
Malostransk\'e n\'am.~25, 118 00~ Praha 1, Czech Republic; and
IEOR, Columbia University, NYC, NY, USA;
%Department of Industrial Engineering and Operations Research, Columbia
%University, 500 West 120th Street, New York City 10027, NY, USA;
\texttt{radoslav@kam.mff.cuni.cz}
}
\and Jan Kyn\v{c}l\thanks{
Department of Applied Mathematics and Institute for Theoretical Computer
Science, Charles University, Faculty of Mathematics and Physics, 
Malostransk\'e n\'am.~25, 118 00~ Praha 1, Czech Republic; and 
Alfr\'ed R\'enyi Institute of Mathematics, Re\'altanoda u. 13-15, Budapest 1053, Hungary;
\texttt{kyncl@kam.mff.cuni.cz}
}
} %end of author

\date{}

\maketitle

\begin{abstract}
In 1958, Hill conjectured that the minimum number of crossings in a drawing of $K_n$ is exactly
$Z(n) = \frac{1}{4} \lfloor\frac{n}{2}\rfloor \left\lfloor\frac{n-1}{2}\right\rfloor \left\lfloor\frac{n-2}{2}\right\rfloor\left\lfloor\frac{n-3}{2}\right\rfloor$. Generalizing the result by \'{A}brego {\it et al.} for 2-page book drawings, we prove this conjecture for plane drawings in which edges are represented by $x$-monotone curves. In fact, our proof shows that the conjecture remains true for $x$-monotone drawings of $K_n$ in which adjacent edges may cross an even number of times, and instead of the crossing number we count the pairs of edges which cross an odd number of times.
We further discuss a generalization of this result to shellable drawings, a notion introduced by \'{A}brego {\it et al.} 
We also give a combinatorial characterization of several classes of $x$-monotone drawings of complete graphs using a small set of forbidden configurations. For a similar local characterization of shellable drawings, we generalize Carath\'eodory's theorem to simple drawings of complete graphs.
\end{abstract}

%========================================================================================================
\section{Introduction}

Let $G$ be a graph with no loops or multiple edges. In a {\em drawing $D$ of a graph $G$} in the plane, the vertices are represented by distinct points and each edge is represented by a simple continuous arc connecting the images of its endpoints. As usual, we identify the vertices and their images, as well as the edges and the arcs representing them. We require that the edges pass through no vertices other than their endpoints. We also assume for simplicity that any two edges have only finitely many points in common, no two edges {\em touch\/} at an interior point and no three edges meet at a common interior point.

A {\em crossing in $D$} is a common interior point of two edges where they properly cross.
%a pair $(p,\{ \alpha, \beta \})$ $p$ is a common interior point of the two arcs $\alpha$ and $\beta$.
The {\em crossing number ${\rm cr}(D)$ of a drawing $D$} is the number of crossings in $D$. The {\em crossing number ${\rm cr}(G)$ of a graph $G$} is the minimum of ${\rm cr}(D)$, taken over all drawings $D$ of $G$. A drawing $D$ is called {\em simple\/} if no two adjacent edges cross and no two edges have more than one common crossing. It is well known and easy to see that every drawing of $G$ which minimizes the crossing number is simple. % has the minimum number of crossings and that a graph $G$ has such drawing which is simple.

According to the famous conjecture of Hill\mycite{guy60,HH63_on_number} (also known as Guy's conjecture), the crossing number of the complete graph $K_n$ on $n$ vertices satisfies ${\rm cr}(K_n) = Z(n)$, where 
\[Z(n) = \frac{1}{4} \bigg\lfloor\frac{n}{2}\bigg\rfloor \left\lfloor\frac{n-1}{2}\right\rfloor \left\lfloor\frac{n-2}{2}\right\rfloor\left\lfloor\frac{n-3}{2}\right\rfloor.\] 
This conjecture has been verified for $n \le 10$ by Guy\mycite{guy72} and recently for $n \le 12$ by Pan and Richter\mydotl\mycite{PR07_K11}\mydotr Moreover for each $n$, there are drawings of $K_n$ with exactly $Z(n)$ crossings\mydotl\mycite{BlaK64_minimal,guy60,HH63_on_number,harborth02}\mydotr
Current best asymptotic lower bound, ${\rm cr}(K_n)\ge 0.8594 Z(n)$, follows from the lower bound on the crossing number of the complete bipartite graph\mycite{KPS07_reduction} by an elementary double-counting argument\mydotl\mycite{RT97_relations}\mydotr

A curve $\alpha$ in the plane is {\em x-monotone\/} if every vertical line intersects $\alpha$ in at most one point.
%A continuous arc $\alpha\subset\mathbb{R}^2$ with the property that every straight-line segment parallel to the $y$-axis intersects $\alpha$ in at most one point is called an {\em $x$-monotone curve}.
A drawing of a graph $G$ in which every edge is represented by an $x$-monotone curve and no two vertices share the same $x$-coordinate is called {\em $x$-monotone\/} (or {\em monotone}, for short). The {\em monotone crossing number\/} $\text{mon-cr}(G)$ {\em of a graph $G$} is the minimum of ${\rm cr}(D)$, taken over all monotone drawings $D$ of $G$.

The {\em rectilinear crossing number $\mathrm{\overline{cr}}(G)$ of a graph $G$} is the smallest number of crossings in a drawing of $G$ where every edge is represented by a straight-line segment. Since every rectilinear drawing of $G$ in which no two vertices share the same $x$-coordinate is $x$-monotone, we have %\{\rm cr}(G) \le \text{mon-cr}(G) \le \mathrm{\overline{cr}}(G)\]
${\rm cr}(G) \le \text{mon-cr}(G) \le \mathrm{\overline{cr}}(G)$
for every graph $G$.

The {\em odd crossing number} $\text{ocr}(G)$ of a graph $G$ is the minimum number of pairs of edges crossing an odd number of times in a drawing of $G$ in the plane. The {\em monotone odd crossing number}, $\text{mon-ocr}(G)$, is the minimum number of pairs of edges crossing an odd number of times in a monotone drawing of $G$. For these two notions of the crossing number, optimal drawings do not have to be simple. Moreover, there are graphs $G$ with $\text{ocr}(G)<\text{cr}(G)$\mycommal\mycite{PSS08_odd,T08_odd_spiders}\mycommar and for every $n$, there is a graph $G$ with $\text{mon-ocr}(G)=1$ and $\text{mon-cr}(G)\ge n$\mydotl\mycite{FPSS13_ht_monotone}\mydotr

We call a drawing of a graph {\em semisimple\/} if adjacent edges do not cross but independent edges may cross more than once. The {\em monotone semisimple odd crossing number\/} of $G$ (called {\em monotone odd $+$} by Schaefer\mycite{S14_survey}),
denoted by $\text{mon-ocr}_+(G)$, is the smallest number of pairs of edges that cross an odd number of times in a monotone semisimple drawing of $G$. We call a drawing of a graph {\em weakly semisimple\/} if every pair of adjacent edges cross an even number of times; independent edges may cross arbitrarily. The {\em monotone weakly semisimple odd crossing number\/} of $G$, denoted by $\text{mon-ocr}_{\pm}(G)$, is the smallest number of pairs of edges that cross an odd number of times in a monotone weakly semisimple drawing of $G$.
Clearly,
%$$\text{mon-ocr}_+(G) \le \text{mon-cr}(G).$$
$\text{mon-ocr}(G) \le \text{mon-ocr}_{\pm}(G) \le \text{mon-ocr}_+(G) \le \text{mon-cr}(G).$

%%%%% old
%The monotone odd $+$ crossing number of $G$, denoted by $\text{mon-ocr}_+(G)$, is the smallest number of pairs of edges that cross an odd number of times in a drawing of $G$ where adjacent edges do not cross but other pairs of edges may cross more than once. Clearly, $$\text{mon-ocr}_+(G) \le \text{mon-cr}(G).$$

The monotone crossing number has been introduced by Valtr\mycite{valtr05_on_pair} and recently further investigated by
%J\'{a}nos
Pach and
%G\'{e}za
T\'{o}th\mycommal\mycite{PT12_monotone}\mycommar who showed that
%the monotone crossing number of a graph is bounded in terms of its crossing number. In particular, they showed that
$\text{mon-cr}(G) < 2{\rm cr}(G)^2$ holds for every graph $G$.
On the other hand, they showed that the monotone crossing number and the crossing number are not always the same: there are graphs $G$ with arbitrarily large crossing numbers such that
%\[\text{mon-cr}(G) \ge \frac{7}{6}{\rm cr}(G)-6.\]
$\text{mon-cr}(G) \ge \frac{7}{6}{\rm cr}(G)-6.$

%there are no graphs with bounded crossing numbers that have arbitrarily large monotone crossing numbers. Specifically, they showed that $\text{mon-cr}(G) < 2{\rm cr}(G)^2$ holds for every graph $G$.
%On the other hand there are also infinitely many graphs $G$ with arbitrarily large crossing numbers such that
%\[\text{mon-cr}(G) \ge \frac{7}{6}{\rm cr}(G)-6.\]

We study the monotone crossing numbers of complete graphs.
The drawings of complete graphs with $Z(n)$ crossings obtained by Bla\v{z}ek and Koman\mycite{BlaK64_minimal} (see also\mycite{harborth02}) are {\em $2$-page book\/} drawings. In such drawings the vertices are placed on a line $l$ and each edge is fully contained in one of the half-planes determined by $l$. Since 2-page drawings may be considered as a strict subset of $x$-monotone drawings, we have $\text{mon-cr}(K_n) \le Z(n)$.

\begin{figure}
 \begin{center}
	\includegraphics{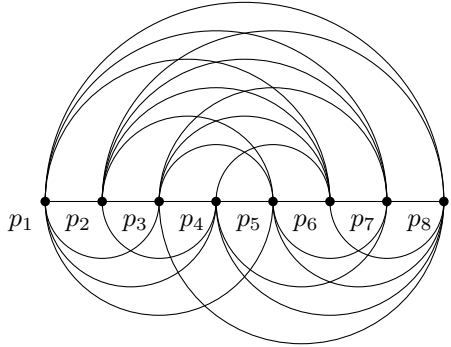}
	\caption{An example of a 2-page book drawing of $K_8$ with $Z(8)=18$ crossings obtained by Bla\v{z}ek and Koman\mydotl\protect\mycite{BlaK64_minimal}\mydotr}
	\label{fig:blazekDrawing}
 \end{center}
\end{figure}

\'{A}brego {\it et al.}\mycite{abre12_2page} recently proved that Hill's conjecture holds for $2$-page book drawings of complete graphs.
We generalize their techniques and show that Hill's conjecture holds for all $x$-monotone drawings of complete graphs, and even for the monotone weakly semisimple odd crossing number.

\begin{theorem}
\label{theorem_1}
For every $n \in \mathbb{N}$, we have
\[\text{\rm mon-ocr}_{\pm}(K_n) = \text{\rm mon-ocr}_+(K_n) = \text{\rm mon-cr}(K_n) = Z(n).\]
\end{theorem}

The rectilinear crossing number of $K_n$ is known to be asymptotically larger than $Z(n)$: this follows from the best current lower bound $\mathrm{\overline{cr}}(K_n) \ge (277/729){{n}\choose{4}} - O(n^3)$\mycite{abre11_halving,abre08_central} and from the simple upper bound $Z(n) \le \frac{3}{8}{n \choose 4}+O(n^3)$.
%and it was obtained by \'{A}brego, Cetina, Fern\'{a}ndez-Merchant, Lea\~{n}os and Salazar\mydotl\mycite{abre11_halving,abre08_central}\mydotr
%Combining Theorem~\ref{theorem_1} with this lower bound we see that the difference between $\mathrm{\overline{cr}}(K_n)$ and $\text{mon-cr}(K_n)$ lies in the asymptotically relevant term, since $\text{mon-cr}(K_n) = Z(n) \le \frac{3}{8}{n \choose 4}+O(n^3)$.

See a recent survey by Schaefer\mycite{S14_survey} for an encyclopedic treatment of all known variants of crossing numbers.

During the preparation of this paper, we were informed that the authors of\mycite{abre12_2page} achieved the result $\text{\rm mon-cr}(K_n) = Z(n)$ already during discussions after their presentation at SoCG 2012 and that Silvia Fernandez-Merchant was going to present it in her keynote talk at LAGOS 2013. The proceedings of the conference were recently published\mydotl\mycite{abre13_more}\mydotr Pedro Ramos\mycite{ram13_egc} then presented the results and some further developments at the XV Spanish Meeting on Computational Geometry (ECG 2013) in his invited talk. Very recently, \'{A}brego {\it et al.}\mycite{aamrs13_shellable} made their paper containing a more general result publicly available.

In Section~\ref{section2}, we first prove Theorem~\ref{theorem_1} for semisimple monotone drawings. Then we extend the result to weakly semisimple monotone drawings, by showing that even crossings of adjacent edges can be easily eliminated in such drawings.

In Section~\ref{section3} we introduce a combinatorial characterization of $x$-monotone drawings of $K_n$. We show that there is a one-to-one correspondence between semisimple, simple or pseudolinear $x$-monotone drawings of $K_n$ and  
%2-colorings of the edges of the complete 3-uniform hypergraph with $n$ labeled vertices
mappings ${[n]\choose 3}\rightarrow \{+,-\}$, called {\em signature functions\/}, avoiding a finite number of certain sub-configurations. The signature functions were introduced by Peters and Szekeres\mycite{SP06_computer_17} as a generalization of order types of planar points sets. 

In Section~\ref{section_+-} we show a further generalization of Theorem~\ref{theorem_1} to shellable drawings and weakly shellable drawings; we define these notions in the beginning of Section~\ref{section_+-}. We show a local characterization of shellable drawings, for which we generalize Caratheodory's theorem to simple drawings of complete graphs. We also show that shellable drawings form a more general class than monotone drawings.
Finally, we further generalize a key lemma from\mycite{abre12_2page}, which implies a generalization of the main result of\mycite{aamrs13_shellable} to weakly semisimple drawings.

In the last section we state our stronger version of Hill's conjecture.

%===============================================================================

\section{Monotone crossing number of the complete graph}\label{section2}

Let $P$ denote a set of $n$ points in the plane in general position and let $k$ be an integer satisfying $0 \le k \le n$. The line segment joining a pair of points $p$ and $q$ in $P$ is a {\em $k$-edge\/} ({\em ${\le}k$-edge}) if there are exactly (at most, respectively) $k$ points of $P$ in one of the open half-planes defined by the line $pq$.

\'{A}brego and Fern\'{a}ndez-Merchant\mycite{abre05_lower_rect} and Lov\'{a}sz {\it et al.}\mycite{LVWW04_quadrilaterals} discovered a relation between the numbers of {\em $k$-edges\/} (or {\em ${\le}k$-edges}) in $P$ and the number of convex $4$-tuples of points in $P$, which is equal to the number of crossings of the complete geometric graph with vertex set $P$. This relation transforms every lower bound on the number of ${\le}k$-edges to a lower bound on the number of crossings. Using this method, many incremental improvements on the rectilinear and pseudolinear crossing number of $K_n$ have been achieved\mydotl\mycite{ABFLS08_extended_lower,abre11_halving,abre05_lower_rect,AGOR07_new_lower,BS06_ksets,LVWW04_quadrilaterals}\mydotr

To prove the lower bound on the $2$-page crossing number of $K_n$, \'{A}brego {\it et al.}\mycite{abre12_2page} generalized the notion of $k$-edges to arbitrary simple drawings of complete graphs. They also introduced the notion of ${\le}{\le}k$-edges, which capture the essential properties of $2$-page book drawings better than ${\le}k$-edges. 
%which had been successfully used before for rectilinear and pseudolinear drawings
We show that the approach using ${\le}{\le}k$-edges can be generalized to arbitrary semisimple $x$-monotone drawings.

For a semisimple drawing $D$ of $K_n$ and distinct vertices $u$ and $v$ of $K_n$, let $\gamma$ be the oriented arc representing the edge $\{u,v\}$. If $w$ is a vertex of $K_n$ different from $u$ and $v$, then we say that $w$ is {\em on the left (right) side of $\gamma$} if the topological triangle $uvw$ with vertices $u$, $v$ and $w$ traced in this order is oriented counter-clockwise (clockwise, respectively). This generalizes the definition introduced by \'{A}brego {\it et al.}\mycite{abre12_2page} for simple drawings. Further generalization is possible for weakly semisimple drawings, where every two edges of the triangle $uvw$ cross an even number of times; see Section~\ref{section_+-}. However, we were not able to find a meaningful generalization of this notion to arbitrary drawings, where the edges of the triangle $uvw$ can cross an odd number of times.

A {\em $k$-edge\/} in $D$ is an edge $\{u,v\}$ of $D$ that has exactly $k$ vertices on the same side (left or right). Since every $k$-edge has $n-2-k$ vertices on the other side, every $k$-edge is also an $(n-2-k)$-edge and so every edge of $D$ is a $k$-edge for some integer $k$ where $0 \le k \le \lfloor n/2 \rfloor -1$.
%%For $n$ even, an $(n/2-1)$-edge is called a \emph{halving edge}, since it has the same number of vertices on each side.
%%%We use a shorthand notation \emph{halving edge} for $k$-edges, where $k = \lfloor n/2 \rfloor -1$.

Analogously to the case of point sets, an $i$-edge in $D$ with $i \le k$ is called a {\em ${\le}k$-edge}. Let $E_i(D)$ be the number of $i$-edges and $E_{{\le}k}(D)$ the number of ${\le}k$-edges of $D$. Clearly, $E_{{\le}k}(D) = \sum_{i=0}^{k}E_{i}(D)$. Similarly,
%a ${\le}j$-edge for some $j \le k$ is called a {\em ${\le}{\le}k$-edge} and
the {\em number $E_{{\le}{\le}k}(D)$ of ${\le}{\le}k$-edges\/} of $D$ is
%the number of ${\le}{\le}k$-edges of $D$, $E_{{\le}{\le}k}(D)$, is
 defined by the following identity.
 \begin{equation}
E_{{\le}{\le}k}(D) = \sum_{j=0}^{k}E_{{\le}j}(D) = \sum_{i=0}^{k}(k+1-i)E_{i}(D).
\label{eq:kkE}
\end{equation}

Considering the only three different simple drawings of $K_4$ up to a homeomorphism of the plane, \'{A}brego {\it et al.}\mycite{abre12_2page} showed that the number of crossings in a simple drawing $D$ of $K_n$ can be expressed in terms of the number of $k$-edges in the following way.
%the following correspondence between the crossing number and the numbers of $k$-edges.

\begin{lemma}[\cite{abre12_2page}]
\label{lemma1}
For every simple drawing $D$ of $K_n$ we have
\begin{equation}
{\rm cr}(D)= 3{{n}\choose{4}} - \sum_{k=0}^{\lfloor n/2 \rfloor - 1}k(n-2-k)E_{k}(D), \label{eq:crossings}
\end{equation}
which can be equivalently rewritten as
\[
{\rm cr}(D) = \;2 \sum_{k=0}^{\lfloor n/2 \rfloor-2} {E_{{\le}{\le}k}(D)}-\frac{1}{2}{n \choose 2}\left\lfloor\frac{n-2}{2}\right\rfloor -\frac{1}{2}\left(1+(-1)^n\right)E_{{\le}{\le}\lfloor n/2 \rfloor-2}(D).
\]

\end{lemma}

Lemma~\ref{lemma1} generalizes the relation found by \'{A}brego and Fern\'{a}ndez-Merchant\mycite{abre05_lower_rect}. We further generalize it to semisimple drawings of $K_n$ where ${\rm cr}(D)$ is replaced by $\text{ocr}(D)$, which counts the number of pairs of edges that cross an odd number of times in $D$.

\begin{lemma}
\label{lemma2}
For every semisimple drawing $D$ of $K_n$ we have
\[
{\rm ocr}(D) = \;2 \sum_{k=0}^{\lfloor n/2 \rfloor-2} {E_{{\le}{\le}k}(D)}-\frac{1}{2}{n \choose 2}\left\lfloor\frac{n-2}{2}\right\rfloor -\frac{1}{2}\left(1+(-1)^n\right)E_{{\le}{\le}\lfloor n/2 \rfloor-2}(D).
\]
\end{lemma}

We recall that a {\em face\/} of a drawing $D$ in the plane is a connected component of the complement of all the edges and vertices of $D$ in $\mathbb{R}^2$. The {\em outer face\/} of $D$ is the unbounded face of $D$.

\begin{proof}[Proof (sketch)]
We just sketch the main idea, which is common with the proof of Lemma~\ref{lemma1}, and then explain the generalization to semisimple drawings. For the details, we refer the reader
to~\cite[Theorem 1 and Proposition 1]{abre12_2page}.

Let $D$ be a semisimple drawing of $K_n$. A {\em separation\/} in $D$ is an unordered triple $\{ab,c,d\}$, where $ab$ is an edge of $D$, $c,d$ are vertices of $D$ distinct from $a,b$, and the orientations of the two triangles $abc$ and $abd$ are opposite. Observe that $\{ab,c,d\}$ is a separation in $D$ if and only if $ab$ is a $1$-edge (and also a {\em halving edge}) in the complete subgraph of $D$ induced by the vertices $a,b,c,d$.
%Let $\mathrm{sep}(D)$ be the total number separations in $D$. 
The total number of separations in $D$ relates to both the crossing number and the numbers of $k$-edges in the following way.

\begin{enumerate}[(i)]
\item Every $k$-edge belongs to exactly $k(n-k-2)$ separations. 
\item Every $4$-tuple of vertices inducing a crossing contributes two separations, and every $4$-tuple of vertices inducing a planar drawing of $K_4$ contributes three separations. In particular, 
%the total number of crossings and separations in every induced subgraph $D$ with $4$ vertices is exactly $3$, which we can write as 
for every complete subgraph $D$ with $4$ vertices we have the equality ${\rm cr}(D)+E_1(D)=3$.
\end{enumerate}

Fact (i) is a direct consequence of the definitions. Fact (ii) is easily seen by inspecting all three homeomorphism classes of simple drawings of $K_4$ in the plane: there is one class with no crossing, and two classes with one crossing, which would form just one class on the sphere; see Figure~\ref{fig_k4ky}.
Lemma~\ref{lemma1} follows from the facts (i) and (ii) by elementary computations.

\begin{figure}
 \begin{center}
   \includegraphics[scale=1]{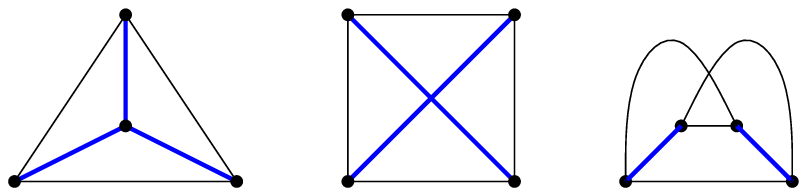}
  \caption{The three homeomorphism classes of simple drawings of $K_4$. The fat edges are $1$-edges.}
 \label{fig_k4ky}
 \end{center}
\end{figure}

To generalize Lemma~\ref{lemma1} to semisimple drawings, we observe that semisimple drawings of $K_4$ can be classified analogously as the simple drawings of $K_4$. In particular, the following claim implies that the equality ${\rm ocr}(D)+E_1(D)=3$ is still satisfied for every semisimple drawing $D$ of $K_4$.

\begin{claim_}
A semisimple drawing $D$ of $K_4$ has at most one pair of edges crossing an odd number of times. Moreover, $D$ has three separations if ${\rm ocr}(D)=0$ and two separations if ${\rm ocr}(D)=1$.
\end{claim_}

In the rest of the proof we prove the claim. 
Let $D$ be a semisimple drawing of $K_4$. Suppose that ${\rm ocr}(D)=0$. Let $abc$ be a triangle in $D$ and let $d$ be the fourth vertex of $D$. See Figure~\ref{fig_k4_semisimple}, left. If the edge $da$ crosses $bc$, then either $d$ and $b$ share no face in the drawing of the subgraph with edges $ab,bc,ad$, or $d$ and $c$ share no face in the drawing of the subgraph with edges $ac,bc,ad$. This means that one of the edges $bd$ or $cd$ either crosses an adjacent edge or crosses another edge an odd number of times. Therefore, the edge $da$ has no crossing with the triangle $abc$. Analogous argument for the edges $db$ and $dc$ shows that $D$ has no crossings at all. In particular, $D$ has three separations; see Figure~\ref{fig_k4ky}, left.

Now suppose that ${\rm ocr}(D)\ge1$ and let $ac$ and $bd$ be two edges that cross an odd number of times. Since all the other edges are adjacent to both $ac$ and $bd$, the vertices $a,b,c,d$ share a common face $F$ in the drawing of the subgraph with edges $ac,bd$. Moreover, the cyclic order of the vertices along the boundary of $F$ is $a,b,c,d$, either clockwise or counter-clockwise. See Figure~\ref{fig_k4_semisimple}, right. 

\begin{figure}
 \begin{center}
   \includegraphics[scale=1]{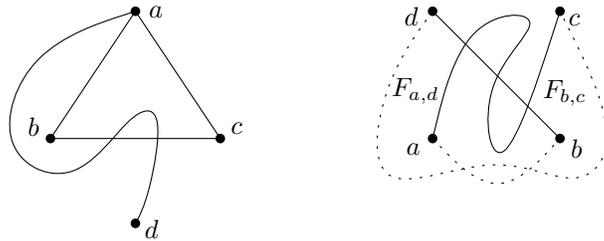}
  \caption{Illustration to the proof of Lemma~\ref{lemma2}.}
 \label{fig_k4_semisimple}
 \end{center}
\end{figure}

We show that at most one more pair of edges can cross, either $ab$ and $cd$, or $ad$ and $bc$, but only an even number of times. For example, in the drawing of the subgraph with edges $ac$, $bd$, $ab$, the vertices $c$ and $d$ belong to the same face, and the edge $cd$ is allowed to cross only the edge $ab$, each time switching faces. If $ab$ and $cd$ cross, then $a$ and $d$ share a unique face $F_{a,d}$ in the drawing of the graph $K$ with edges $ac,bd,ab,cd$, and $c$ and $b$ share a unique face $F_{b,c}$ different from $F_{a,d}$. Since the edges $ad$ and $bc$ are adjacent to all edges of $K$, the edge $ad$ lies completely in $F_{a,d}$, the edge $bc$ lies completely in $F_{b,c}$ and thus $ad$ and $bc$ cannot cross. A symmetric argument shows that if $ab$ and $cd$ are disjoint, then $ad$ and $bc$ are either disjoint or cross an even number of times.
In any case, we have ${\rm ocr}(D)\le 1$ (and the pair crossing number of $D$ is at most $2$).

It remains to show that every semisimple drawing $D$ of $K_4$ with ${\rm ocr}(D)=1$ has exactly two $1$-edges. More precisely, we show that the two $1$-edges always form a perfect matching.

Let $e$ be an edge in $D$ incident with the outer face. An {\em edge flip\/} is an operation where the portion of $e$ incident with the outer face is redrawn along the other side of the drawing; see Figure~\ref{fig_edge_flip}. For drawings on the sphere, the edge flip is just a homeomorphism of the sphere. For every bounded face $F$ of $D$, there is a sequence of edge flips that makes $F$ the outer face.

\begin{figure}
 \begin{center}
   \includegraphics[scale=1]{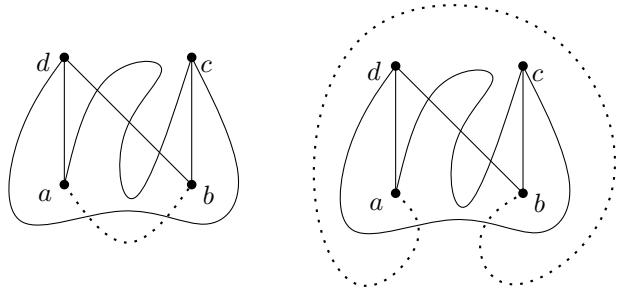}
  \caption{An edge flip of $ab$.}
 \label{fig_edge_flip}
 \end{center}
\end{figure}

If $D$ is a semisimple drawing of $K_4$, then every edge flip of an edge $e$ changes the orientation of the two triangles adjacent to $e$. Consequently, exactly the four edges adjacent to $e$, forming a $4$-cycle, change from $1$-edges to $0$-edges or vice versa. Also observe that the edge flip of $e$ can be performed only if $e$ is a $0$-edge. It follows that $1$-edges form a perfect matching in $D$ if and only if they form a perfect matching in the drawing obtained by the edge flip.

Let $D$ be a semisimple drawing of $K_4$ with ${\rm ocr}(D)=1$. Let $ac$ and $bd$ be the two edges that cross an odd number of times. By performing edge flips, we may assume that all the vertices are adjacent to the outer face of the drawing of the subgraph $H$ with edges $ac$ and $bd$. Each edge $e$ of the remaining four edges can be drawn in two essentially different ways with respect to $H$, which differ just by an edge flip of $e$ in $H+e$; see Figure~\ref{fig_edge_flip}. In total, there are $16$ possible combinations. We cannot, however, assume any particular combination, since not all edge flips are always available.
Observe that the orientations of all triangles are determined by the four binary choices for the edges $ab,bc,cd,ad$. Also, changing the choice for one edge $e$ has the same effect on the orientations of the triangles as the edge flip of $e$. For one particular choice, for example the one yielding the middle drawing in Figure~\ref{fig_k4ky}, the $1$-edges form a perfect matching. Changing the choice for a subset of edges yields either a perfect matching of $1$-edges or a complete graph of $1$-edges. However, the latter option is excluded by the fact that in every semisimple drawing the edges incident with the outer face are $0$-edges. This finishes the proof of the claim and the lemma. 
%
%%%earlier short version
%%The main reason is that the cycle $C_4$ cannot be drawn in the plane in such a way that both its pairs of opposite edges cross oddly while adjacent edges do not cross. Hence, every 4-tuple of vertices can account for at most one odd crossing pair, and 
%%``combinatorially'' corresponds to $K_4$ drawn as simple topological graph. 
%%%%
\end{proof}

Considering ${\le}k$-edges, \'{A}brego and Fern\'{a}ndez-Merchant\mycite{abre05_lower_rect} and Lov\'{a}sz {\it et al.}\mycite{LVWW04_quadrilaterals} proved that for rectilinear drawings of $K_n$, the inequality $E_{{\le}k} \ge 3{{k+2}\choose{2}}$ together with (\ref{eq:crossings}) gives $\mathrm{\overline{cr}}(G) \ge Z(n)$. However, there are simple $x$-monotone (even $2$-page) drawings of $K_n$ where $E_{{\le}k} < 3{{k+2}\choose{2}}$ for $k=1$\mydotl\mycite{abre12_2page}\mydotr \'{A}brego {\it et al.}\mycite{abre12_2page}
showed that the inequality $E_{{\le}{\le}k} \ge 3{{k+3}\choose{3}}$, which is implied by inequalities $E_{{\le}j} \ge 3{{j+2}\choose{2}}$ for $j\le k$, is satisfied by all $2$-page book drawings. We show that the same inequality is satisfied by all $x$-monotone semisimple drawings of $K_n$.

Let $\{v_1,v_2,\ldots,v_n\}$ be the vertex set of $K_n$.
Note that we can assume that all vertices in an $x$-monotone drawing lie on the $x$-axis.
%%%Throughout the section we denote the point representing the vertex $v_i$, for $1 \le i \le n$, in $D$ as $v_i$ and
We also assume that the $x$-coordinates of the vertices satisfy $x(v_1)<x(v_2)<\cdots<x(v_n)$.

The following observation describes the structure of $k$-edges incident to vertices on the outer face in semisimple drawings of complete graphs. See Figure~\ref{fig_outer_vertex}, left.

\begin{observation}\label{obs_outer}
Let $D$ be a semisimple drawing of $K_n$, not necessarily $x$-mono\-to\-ne. Let $v$ be a vertex incident to the outer face of $D$ and let $\gamma_i$ be the $i$th edge incident to $v$ in the counter-clockwise order so that $\gamma_1$ and $\gamma_{n-1}$ are incident to the outer face in a small neighborhood of $v$. Let $v_{k_i}$ be the other endpoint of $\gamma_i$. Then for every $i,j$, $1\le i<j \le n-1$, the triangle $v_{k_i}vv_{k_j}$ is oriented clockwise. Consequently, for every $k$ with $1\le k\le (n-1)/2$, the edges $\gamma_k$ and $\gamma_{n-k}$ are $(k-1)$-edges.
\end{observation}

\begin{figure}
 \begin{center}
   \includegraphics[scale=1]{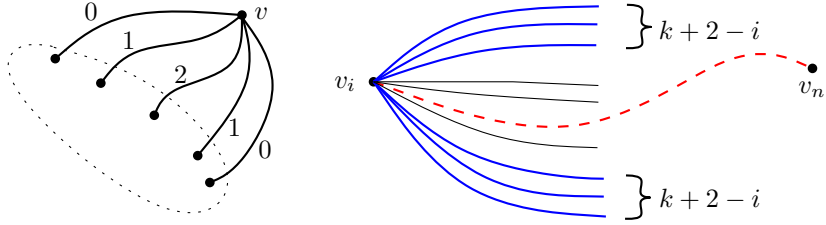}
  \caption{Left: $k$-edges incident with a vertex on the outer face. Right: After removing $v_n$, at least $k+2-i$ right edges at $v_i$ are invariant ${\le}k$-edges.}
 \label{fig_outer_vertex}
 \end{center}
\end{figure}

For an $x$-monotone drawing $D$ of $K_n$, we use Observation~\ref{obs_outer} for the vertex $v_n$ and the drawing $D$ and then for each $i$, for the vertex $v_i$ and the drawing of the subgraph induced by $v_i, v_{i+1}, \dots, v_n$.

The following definitions were introduced by \'{A}brego {\it et al.}\mycite{abre12_2page} for $2$-page book drawings. Let $D$ be a semisimple $x$-monotone drawing of $K_n$ and let $D'$ be the drawing obtained from $D$ by deleting the vertex $v_n$ together with its adjacent edges. A $k$-edge in $D$ is a {\em $(D,D')$-invariant $k$-edge\/} if it is also a $k$-edge in $D'$. It is easy to see that every ${\le}k$-edge in $D'$ is also a ${\le}(k+1)$-edge in $D$.
If $0\le j \le k \le \lfloor n/2 \rfloor - 1$, then a $(D,D')$-invariant $j$-edge  is called a {\em $(D,D')$-invariant ${\le}k$-edge}. Let $E_{{\le}k}(D,D')$ denote the number of $(D,D')$-invariant ${\le}k$-edges.

For $i<j$, the edge $v_iv_j$ is called a {\em right edge at $v_i$}. The right edges at $v_i$ have a natural vertical order, which coincides with the order of their crossings with an arbitrary vertical line separating $v_i$ and $v_{i+1}$. The set of $j$ {\em topmost\/} ({\em bottommost}) right edges at $v_i$ is the set of $j$ right edges at $v_i$ that are above (below, respectively) all other right edges at $v_i$ in their vertical order.

\begin{lemma}\label{lemma_o_vnitrnich_vrcholech}
Let $D$ be a semisimple $x$-monotone drawing of $K_n$ and let $k$ be a fixed integer such that $0\le k \le (n-3)/2$. For every $i\in\{1,2,\dots,k+1\}$, the ${k+2-i}$ bottommost and the ${k+2-i}$ topmost right edges at $v_i$ are $\le k$-edges in $D$.
Moreover, at least ${k+2-i}$ of these $\le k$-edges are $(D,D')$-invariant ${\le}k$-edges.
\end{lemma}

\begin{proof}
See Figure~\ref{fig_outer_vertex}, right. The first part of the lemma follows directly from Observation~\ref{obs_outer}. If the edge $v_iv_n$ is one of the ${k+2-i}$ topmost right edges at $v_i$, then the ${k+2-i}$ bottommost right edges at $v_i$ are $(D,D')$-invariant ${\le}k$-edges. Otherwise the ${k+2-i}$ topmost right edges at $v_i$ are $(D,D')$-invariant ${\le}k$-edges. 

\end{proof}

\begin{corollary}\label{cor_invariant}
We have
\[E_{{\le}k}(D,D') \ge \sum_{i=1}^{k+1}(k+2-i) = {{k+2}\choose{2}}.\]
\vskip -0.6cm
\end{corollary}

The following theorem gives a lower bound on the number of ${\le}{\le}k$-edges. The proof is essentially the same as in\mycite{abre12_2page}, we only extracted Lemma~\ref{lemma_o_vnitrnich_vrcholech}, which needed to be generalized. Together with Lemma~\ref{lemma2}, Theorem~\ref{thm2} yields the second and the third equality in Theorem~\ref{theorem_1}, %for semisimple drawings.
by the same computation as in\mydotl\mycite{abre12_2page}\mydotr

\begin{theorem}
\label{thm2}
Let $n \ge 3$ and let $D$ be a semisimple $x$-monotone drawing of $K_n$. Then for every $k$ satisfying $0 \le k < n/2 - 1$, we have
%\[E_{{\le}{\le}k}(D) \ge 3{{k+3} \choose {3}}.\]
$E_{{\le}{\le}k}(D) \ge 3{{k+3} \choose {3}}.$
\end{theorem}

\begin{proof}
The proof proceeds by induction on $n$ and $k$ starting at $n=3$ and $k=-1$. The case $n=3$ is trivially true, and the case $k=-1$ is taken care of by setting $E_{{\le}{\le}-1}(D)=0$ for every drawing $D$. Let $n \ge 4$ and let $D$ be a semisimple $x$-monotone drawing of $K_n$. For the induction step we remove the point $v_n$ together with its adjacent edges to obtain a drawing $D'$ of $K_{n-1}$, which is also semisimple and $x$-monotone.

Using Observation~\ref{obs_outer} we see that, for $0 \le i \le k < n/2 - 1$, there are two $i$-edges adjacent to $v_n$ in $D$ and together they contribute with
%\[2\sum_{i=0}^{k}(k+1-i)=2{{k+2} \choose {2}}\]
$2\sum_{i=0}^{k}(k+1-i)=2{{k+2} \choose {2}}$
to $E_{{\le}{\le}k}(D)$ by~(\ref{eq:kkE}).

Let $\gamma$ be an $i$-edge in $D'$. If $i\le k$, then $\gamma$ contributes with $(k-i)$ to the sum
\[E_{{\le}{\le}k-1}(D')= \sum_{i=0}^{k-1}(k-i)E_{i}(D').\]
%$E_{{\le}{\le}k-1}(D')= \sum_{i=0}^{k-1}(k-i)E_{i}(D').$
We already observed that $\gamma$ is either an $i$-edge or an $(i+1)$-edge in $D$. If $\gamma$ is also an $i$-edge in $D$ (that is, $\gamma$ is a $(D,D')$-invariant $i$-edge), then it contributes with $(k+1-i)$ to $E_{{\le}{\le}k}(D)$. This is a gain of +1 towards $ E_{{\le}{\le}k-1}(D')$.  If $\gamma$ is an $(i+1)$-edge in $D$, then it contributes only with $(k-i)$ to $E_{{\le}{\le}k}(D)$.
Therefore we have 
\[E_{{\le}{\le}k}(D) = 2{{k+2}\choose{2}} + E_{{\le}{\le}k-1}(D') + E_{{\le}k}(D,D').\]

By the induction hypothesis we know that $E_{{\le}{\le}k-1}(D') \ge 3{{k+2} \choose {3}}$  and thus we obtain
\[E_{{\le}{\le}k}(D) \ge 3{{k+3} \choose {3}} - {{k+2} \choose {2}} + E_{{\le}k}(D,D').\]

The theorem follows by plugging the lower bound from Corollary~\ref{cor_invariant}.  
\end{proof}

%---------------------------------------------------------------------------------
%Removing even adjacent crossings from monotone weakly semisimple drawings

\subsection{Removing even adjacent crossings} 
Here we finish the proof of Theorem~\ref{theorem_1} by showing that allowing adjacent edges to cross evenly yields no substantially new monotone drawings of $K_n$.

The {\em rotation\/} at a vertex $v$ in a drawing is the clockwise cyclic order of the neighbors of $v$ in which the corresponding edges appear around $v$. The {\em rotation system\/} of a drawing is the set of rotations of all its vertices.

\begin{proposition}\label{proposition_prekreslovaci}
Let $D$ be a weakly semisimple monotone drawing of $K_n$. Then there is a semisimple monotone drawing $D'$ of $K_n$ such that for every two edges $e,f$ of $K_n$, the parity of the number of crossings between $e$ and $f$ in $D'$ is the same as in $D$. Moreover, $D'$ and $D$ have the same rotation system and the same above/below relations of vertices and edges.
\end{proposition}

\begin{proof}
Let $O(D)$ be the set of pairs of edges of $K_n$ that cross an odd number of times in $D$. 
Let $D'$ be a weakly semisimple monotone drawing of $K_n$ with minimum total number of crossings such that $D'$ is {\em strongly equivalent\/} to $D'$, that is, $D'$ and $D$ have the same rotation system, the same above/below relations of vertices and edges and $O(D')=O(D)$. 
We show that $D'$ is semisimple.

Suppose for contrary that $D'$ has two adjacent edges $e,f$ that cross. Since $D'$ is weakly semisimple, $e$ and $f$ cross at least twice. Let $v$ be the common vertex of $e$ and $f$ and suppose that $e$ is above $f$ in the neighborhood of $v$. Let $x_1$ and $x_2$ be the two crossings of $e$ and $f$ closest to $v$. See Figure~\ref{fig_weakly_bigon}, left. Let $B$ be the closed topological disc bounded by the two portions of $e$ and $f$ between $x_1$ and $x_2$. Clearly, $B$ has no vertex on its boundary. Moreover, we claim that $B$ has no vertex in its interior. For if $B$ contains a vertex $w$ in its interior, then $w$ is below $f$ and above $e$. This implies that the edge $vw$ is below $f$ and above $e$ in the neighborhood of $v$, which is absurd.

\begin{figure}
 \begin{center}
   \includegraphics[scale=1]{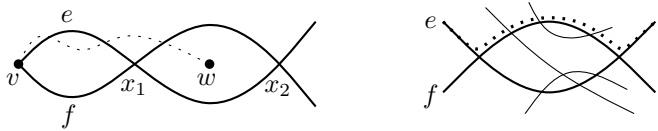}
  \caption{Left: the edge $vw$ is forced to cross $e$ or $f$ an odd number of times. Right: decreasing the total number of crossings.}
 \label{fig_weakly_bigon}
 \end{center}
\end{figure}

Since $B$ contains no vertices, every edge other than $e$ and $f$ crosses the boundary of $B$ an even number of times. Therefore, by redrawing an open segment of $e$ or $f$ containing $x_1$ and $x_2$ along the other side of $B$, we obtain a drawing strongly equivalent to $D'$ with at most $\text{cr}(D')-2$ crossings. See Figure~\ref{fig_weakly_bigon}, right.
\end{proof}

We note that using slightly more careful redrawing operations (such as those in the proof of Theorem~\ref{theorem_classif_pseudolin} in Section~\ref{section3}), we may obtain a semisimple monotone drawing $D''$ strongly equivalent to $D$ such that for every two edges, the number of their common crossings in $D''$ is not larger than in $D$.

By Proposition~\ref{proposition_prekreslovaci}, the odd crossing number of a weakly semisimple monotone drawing of $K_n$ is equal to the odd crossing number of some semisimple monotone drawing of $K_n$. This proves the first equality in Theorem~\ref{theorem_1}.

%==========================================================================================

\section{Combinatorial description of monotone drawings}\label{section3}

In this section we develop a combinatorial characterization of $x$-monotone drawings based on the signature functions introduced by Peters and Szekeres\mycite{SP06_computer_17} as generalizations of order types of planar point sets. Let $T_n$ be the set of ordered triples $(i,j,k)$ with $i<j<k$, of the set $[n] = \{1,2,\ldots,n\}$ and let $\Sigma_n$ be the set of {\em signature functions\/} $\sigma \colon T_n \to \{-,+\}$. The set $T_n$ may be also regarded as the set ${[n]\choose 3}$ of all unordered triples, since we write all the triples in the increasing order of their elements. 

Let $D$ be an $x$-monotone drawing of the complete graph $K_n=(V,E)$ with vertices $v_1,v_2,\ldots,v_n$ such that their $x$-coordinates satisfy $x(v_1) < x(v_2) < \cdots < x(v_n)$. We assign a signature function $\sigma \in \Sigma_n$ to the drawing $D$ according to the following rule. For every edge $e=v_iv_k \in E$ and every integer $j\in(i,k)$, let $\sigma(i,j,k) = -$ if the point $v_j$ lies above the arc representing the edge $e$ and $\sigma(i,j,k)=+$ otherwise. See Figure~\ref{fig:sigma}. Note that if the drawing $D$ is also semisimple, then a triangle $v_iv_jv_k$, with $j \in (i,k)$, is oriented clockwise (counter-clockwise) if and only if $\sigma(i,j,k)=-$ ($\sigma(i,j,k)=+$, respectively).

\begin{figure}
 \begin{center}
  \includegraphics{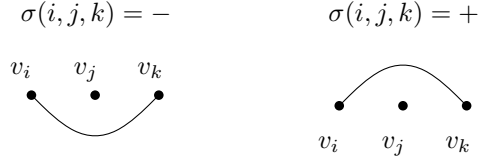}
  \caption{The negative and the positive signature $\sigma(i,j,k)$.}
  \label{fig:sigma}
 \end{center}
\end{figure}

It is easy to see that, for every signature function $\sigma \in \Sigma_n$, there exists an $x$-monotone drawing $D$ which induces $\sigma$. However, some signature functions are induced only by drawings that are not semisimple.
% a semisimple drawing that induces $\sigma$ does not always exist. 
We show a characterization of simple and semisimple $x$-monotone drawings by small forbidden configurations in the signature functions.

For integers $a,b,c,d \in [n]$ with $a < b < c < d$, signs $\xi_1,\xi_2,\xi_3,\xi_4 \in \{-,+\}$ and a signature function $\sigma \in \Sigma_n$,
 we say that the 4-tuple $(a,b,c,d)$ {\em is of the form $\xi_1\xi_2\xi_3\xi_4$ in $\sigma$} if 
\[\sigma(a,b,c)=\xi_1,\; \sigma(a,b,d)=\xi_2,\; \sigma(a,c,d)=\xi_3, \text{ and } \; \sigma(b,c,d)=\xi_4.\]
Alternatively, we write 
%$\sigma(a,b,c,d)=\xi_1\xi_2\xi_3\xi_4$, or also
$\sigma(\{\pi(a),\pi(b),\pi(c),\pi(d)\})=\xi_1\xi_2\xi_3\xi_4$ for any permutation $\pi$ of the set $\{a,b,c,d\}$.

For a sign $\xi \in \{-,+\}$ we use $\overline{\xi}$ to denote the opposite sign, that is, if $\xi = +$ then $\overline{\xi}=-$ and conversely, if $\xi = -$ then $\overline{\xi}=+$.

\subsection{Simple and semisimple $x$-monotone drawings}

\begin{theorem}
\label{theorem_classif_semisimple}
A signature function $\sigma \in \Sigma_n$ can be realized by a semisimple $x$-monotone drawing if and only if every $4$-tuple of indices from $[n]$ is of one of the forms
\begin{align*}
 {+}{+}{+}{+},&{-}{-}{-}{-},{+}{+}{-}{-},{-}{-}{+}{+},{-}{+}{+}{-},{+}{-}{-}{+}, \\
& {-}{-}{-}{+},{+}{+}{+}{-},{+}{-}{-}{-},{-}{+}{+}{+}
\end{align*}
in $\sigma$. The signature function $\sigma$ can be realized by a simple $x$-monotone drawing if, in addition,  there is no $5$-tuple $(a,b,c,d,e)$ with $a<b<c<d<e$ such that \[\sigma(a,b,e)=\sigma(a,d,e)=\sigma(b,c,d)=\overline{\sigma(a,c,e)}.\]
\end{theorem}

See Figure~\ref{fig_semisimple} and Figure~\ref{fig_forb_5tuple_0} for an illustration of the first and the second part of the theorem.

\begin{proof}
Let $\sigma$ be a signature function with a {\em forbidden $4$-tuple}, that is, an ordered $4$-tuple $(a,b,c,d)$ whose form is not listed in the statement of the theorem. Such a $4$-tuple $(a,b,c,d)$ is one of the forms ${\xi_1}{\overline{\xi_1}}{\xi_1}{\xi_2}$ or ${\xi_2}{\xi_1}{\overline{\xi_1}}{\xi_1}$ where $\xi_1, \xi_2 \in \{-,+\}$. If $(a,b,c,d)$ is of the form ${+}{-}{+}{\xi}$ where $\xi \in \{-,+\}$ is an arbitrary sign, then the edges $v_av_c$ and $v_av_d$ are forced to cross between the vertical lines going through $v_b$ and $v_c$; see Figure~\ref{fig_ac_ad_cross}. But this is not allowed in a semisimple drawing and we have a contradiction. The other cases are symmetric.

\begin{figure}
 \begin{center}
   \includegraphics[scale=1]{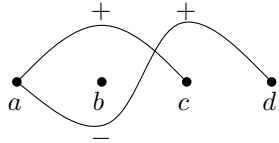}
  \caption{A $4$-tuple $(a,b,c,d)$ of the form ${+}{-}{+}{\xi}$ forces two adjacent edges to cross.}
 \label{fig_ac_ad_cross}
 \end{center}
\end{figure}

On the other hand, let $\sigma$ be a signature function such that every 4-tuple is of one of the ten allowed forms in $\sigma$. We will construct a semisimple $x$-monotone drawing $D$ of $K_n$ which induces $\sigma$. We use the points $v_i=(i,0)$, $i \in [n]$, as vertices and connect consecutive pairs of vertices by straight-line segments.

For $m \in [n]$, let $L_m$ be the vertical line containing $v_m$. In every $x$-monotone drawing, the line $L_m$ intersects every edge $\{v_i,v_j\}$ with $1\le i < m\le j \le n$ exactly once. 
To draw the edges of $K_n$, it suffices to specify the positions of their
intersections with the lines $L_m$ and to draw the edges as polygonal lines with bends at these intersections.
Instead of the absolute position of these intersections on $L_m$, we only need to determine their vertical total ordering, which we represent by a total ordering $\prec_m$ of the corresponding edges. The edges whose right endpoint is $v_m$ will be ordered by $\prec_m$ according to their vertical order in the left neighborhood of $v_m$. The edges with left endpoint $v_m$ are not considered in $\prec_m$.

The idea of the construction is to interpret the signature function as the set of above/below relations for vertices and edges and take a set of orderings $\prec_m$ that obey these relations and minimize the total number of crossings.
In the rest of the proof we show a detailed, explicit construction of the orderings $\prec_m$ which induce an $x$-monotone semisimple drawing.

For $i \in [n]$, we define an ordering $\lessdot_i$ of the edges with a common left endpoint $v_i$ (that is, the right edges at $v_i$) in the following way. If $e=\{v_i,v_j\}$ and $f=\{v_i,v_k\}$, $i<j,k$, are two such edges, then we set $e \lessdot_i f$ if either $j<k$ and $\sigma(i,j,k)=+$, or $k<j$ and $\sigma(i,k,j)=-$. Clearly, the relation $\lessdot_i$ is irreflexive, antisymmetric and for every two right edges $e,f$ at $v_i$ either $e \lessdot_i f$ or $f \lessdot_i e$. 
To show that $\lessdot_i$ is a total ordering, it remains to prove that it is transitive.
%%%forbidden 4-tuple in $\sigma$, then $\lessdot_i$ is transitive and hence a total ordering. To verify transitivity, 
Suppose for contrary that there are three edges $e=\{v_i,v_j\}$, $f=\{v_i,v_k\}$ and $g=\{v_i,v_l\}$ with $i<j<k<l$ such that $e \lessdot_i f$, $f \lessdot_i g$ and $g \lessdot_i e$. Then $\sigma(i,j,k)=+$, $\sigma(i,k,l)=+$ and $\sigma(i,j,l)=-$, so the 4-tuple $i,j,k,l$ is of the form ${+}{-}{+}\xi$, which is forbidden.
Similarly, if $f \lessdot_i e$, $e \lessdot_i g$ and $g \lessdot_i f$, then the 4-tuple $i,j,k,l$ is of the form ${-}{+}{-}\xi$, which is forbidden as well.

We proceed by induction on $m$. In the case $m=1$ the ordering $\prec_1$ is empty. For $m=2$ the ordering $\prec_2$ compares only edges with the common endpoint $v_1$, so we can set ${\prec_2}={\lessdot_1}$. Since all the edges are drawn by line segments starting in a common endpoint, no crossings appear between $L_1$ and $L_2$.

Let $m>2$. For the inductive step we consider the following sets $S_1,\ldots,S_6$ of edges which intersect $L_{m-1}$ and $L_{m}$ (see Figure~\ref{fig:orderingEdges}):

\begin{align*}
&S_1=\{\{v_i,v_j\} \mid \sigma(i,m-1,j)=-, \sigma(i,m,j)=- \},\\
&S_2=\{\{v_{m-1},v_j\} \mid \sigma(m-1,m,j)=- \},\\
&S_3=\{\{v_i,v_j\} \mid \sigma(i,m-1,j)=+, \sigma(i,m,j)=- \textrm{ or } j=m \},\\
&S_4=\{\{v_i,v_j\} \mid \sigma(i,m-1,j)=-, \sigma(i,m,j)=+ \textrm{ or } j=m \},\\
&S_5=\{\{v_{m-1},v_j\} \mid \sigma(m-1,m,j)=+ \}, \\
&S_6=\{\{v_i,v_j\} \mid \sigma(i,m-1,j)=+, \sigma(i,m,j)=+ \}.
\end{align*}
The edges within sets $S_2$ and $S_5$ are ordered according to $\lessdot_{m-1}$ and the edges in each of the remaining sets $S_k$ according to $\prec_{m-1}$. For $e \in S_k$ and $f \in S_\hell$ where $k < \hell$, we set $e \prec_m f$. Observe that $\prec_m$ is a total ordering.

\begin{figure}
	\centering
	\includegraphics[scale=1]{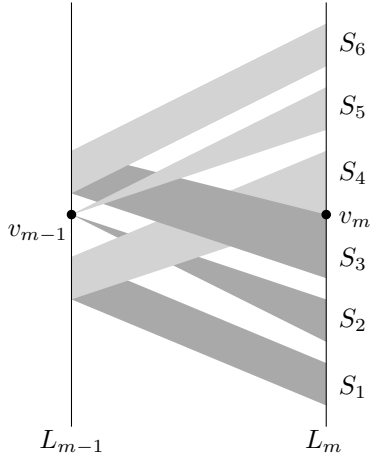}
	\caption{Placing edges and minimizing the number of crossings.}
	\label{fig:orderingEdges}
\end{figure}

We show that the drawing $D$ determined by the orders $\prec_m$ is semisimple. Suppose for contradiction that two adjacent edges $e=\{v_i,v_j\}$ and $f=\{v_i,v_k\}$, with $i<j,k$ and $e \lessdot_i f$, cross. Their leftmost crossing occurs between lines $L_{m-1}$ and $L_m$, where $i < m-1$ and $m\le j,k$. There are three cases: 
\begin{enumerate}[(i)]
\item $e \in S_6$ and $f \in S_3$, 
\item $e \in S_4$ and $f \in S_1$, or
\item $e \in S_4$ and $f \in S_3$.
\end{enumerate}

We analyze the cases (i) and (iii) together, case (i) and case (ii) are symmetric. If $j<k$ then $\sigma(i,m,k)= -$ and by the definition of the relation $\lessdot_i$, we have $\sigma(i,j,k)= +$. This further implies that $m<j$ and $\sigma(i,m,j)=+$. Thus $(i,m,j,k)$ forms a forbidden 4-tuple. If $k<j$, then $\sigma(i,m,j)= +$, $\sigma(i,k,j)= -$, which implies that $m<k$ and $\sigma(i,m,k)= -$, and so we obtain a forbidden $4$-tuple $(i,m,k,j)$. 

Now suppose that two adjacent edges $e=\{v_i,v_k\}$ and $f=\{v_j,v_k\}$, with $i,j<k$, cross. Their leftmost crossing occurs between lines $L_{m-1}$ and $L_m$, where $i,j \le m-1$ and $m<k$. We may assume that $f\prec_m e$ and $e\prec_{m-1} f$.
There are five cases: 
\begin{enumerate}[(i)]
\item $e \in S_6$ and $f \in S_3$, 
\item $e \in S_4$ and $f \in S_1$,
\item $e \in S_4$ and $f \in S_3$,
\item $e \in S_4$ and $f \in S_2$, or
\item $e \in S_5$ and $f \in S_3$.
\end{enumerate}

Case (i) and case (ii) are symmetric, as well as case (iv) and case (v). Therefore it is sufficient to consider cases (i), (iii) and (v). In all these three cases $\sigma(j,m,k)= -$ and $\sigma(i,m,k)= +$. If $j<i$, then $\sigma(j,i,k)= +$ since $e\prec_{m-1} f$ and the edges $e$ and $f$ do not cross to the left of $L_{m-1}$. Hence $(j,i,m,k)$ forms a forbidden $4$-tuple. If $i<j$, then analogously $\sigma(i,j,k)=-$ and $(i,j,m,k)$ forms a forbidden $4$-tuple.
This finishes the proof that $D$ is semisimple.

It remains to show the second part of the theorem.
If $D$ is a drawing with a signature function $\sigma$ with a {\em forbidden $5$-tuple\/} $(a,b,c,d,e)$, then $D$ is not simple as the edges $v_av_e$ and $v_bv_d$ are forced to cross at least twice; see Figure~\ref{fig_forb_5tuple_0}.

\begin{figure}
 \begin{center}
   \includegraphics[scale=1]{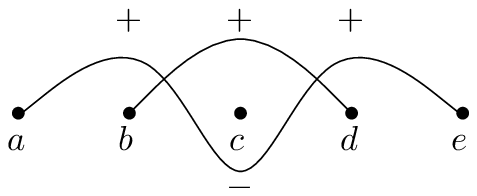}
  \caption{A forbidden $5$-tuple $(a,b,c,d,e)$ forces at least two crossings between $v_av_e$ and $v_bv_d$.}
 \label{fig_forb_5tuple_0}
 \end{center}
\end{figure}

In the rest of the proof we show the second part of the theorem.
%that semisimple $x$-monotone drawings that are not simple always contain a forbidden $5$-tuple.

Given a signature function $\sigma$ with no forbidden 4-tuples and 5-tuples we apply the same construction as before to obtain a semisimple $x$-monotone drawing $D$. We show that $D$ is, in addition, simple. Since $D$ is semisimple, no two crossing edges have an endpoint in common. By the construction of $D$, every crossing $c$ of two edges $e$ and $f$ occurs between lines $L_{m}$ and $L_{m+1}$ for some $m \in [n-1]$ and we say that 
%%$v_{m}$ is the {\em left} and 
$v_{m+1}$ is the {\em right neighbor\/} of $c$. The right neighbor is either an endpoint of $e$ or $f$ or it separates the crossings of $L_{m+1}$ with $e$ and $f$. Suppose that there are edges $e=v_iv_j$ and $f=v_kv_\hell$ with $i<k<j,\hell$ that cross at least twice. We show that then there is always a forbidden 4-tuple or a forbidden 5-tuple in $\sigma$.

Let $v_m$ be the right neighbor of the leftmost crossing and $v_{m'}$ the right neighbor of the second leftmost crossing of $e$ and $f$. Observe that $i,k<m<m'\le j,l$.

First assume that $\hell<j$. Refer to Figure~\ref{fig_forb_5tuple_1}. If $\sigma(i,k,j)=\sigma(i,\hell,j)=\xi$ for some $\xi \in \{-,+\}$, then  $\xi=\sigma(k,m,\hell)=\overline{\sigma(i,m,j)}$ and so $(i,k,m,\hell,j)$ forms a forbidden 5-tuple. If $\sigma(i,k,j)=\overline{\sigma(i,\hell,j)}=\xi$ for some $\xi \in \{-,+\}$, then $e$ and $f$ cross at least three times and so $m'<\hell,j$. We have $\xi=\sigma(k,m,\hell)=\overline{\sigma(i,m,j)}=\overline{\sigma(k,m',\hell)}=\sigma(i,m',j)$. %Moreover, since $e$ and $f$ cross an odd number of times, they cross at least three times and, in particular, $m<m'$. 
If $\sigma(k,m,m')=\overline{\xi}$, then $(k,m,m',\hell)$ forms a forbidden $4$-tuple. If $\sigma(k,m,m')=\xi$, then $(i,k,m,m',j)$ forms a forbidden $5$-tuple.

%%and $\xi=\sigma(i,m',j)=\overline{\sigma(k,m',\hell)}$. If $\sigma(m,m',\hell)=\xi$, then $k<m<m'<\hell$ induce a forbidden 4-tuple. If $\sigma(m,m',\hell)=\overline{\xi}$, then $i<m<m'<\hell<j$ induce a forbidden 5-tuple.
%%

\begin{figure}
 \begin{center}
   \includegraphics[scale=1]{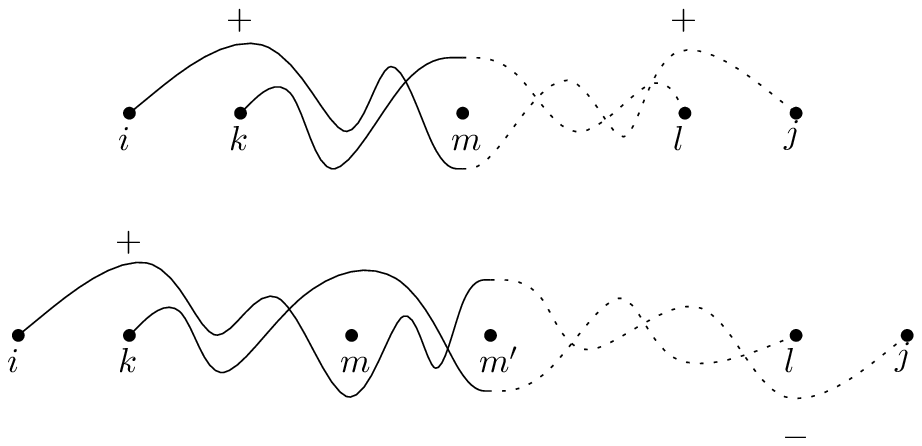}
  \caption{Edges $v_iv_j$ and $v_kv_{\hell}$ crossing twice imply a forbidden $5$-tuple or $4$-tuple; case $\hell<j$.}
 \label{fig_forb_5tuple_1}
 \end{center}
\end{figure}

\begin{figure}
 \begin{center}
   \includegraphics[scale=1]{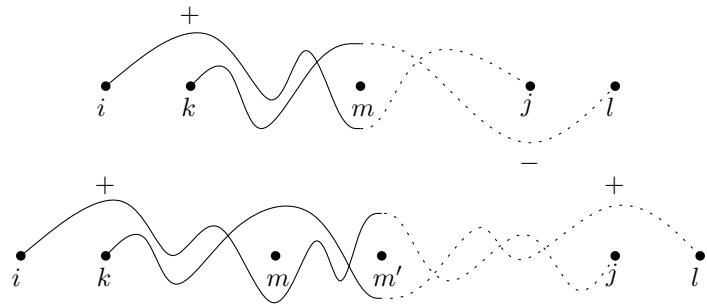}
  \caption{Edges $v_iv_j$ and $v_kv_{\hell}$ crossing twice imply a forbidden $5$-tuple or $4$-tuple; case $j<\hell$.}
 \label{fig_forb_5tuple_2}
 \end{center}
\end{figure}

Conversely let $j<\hell$. Refer to Figure~\ref{fig_forb_5tuple_2}. Assume that $\sigma(i,k,j)=\overline{\sigma(k,j,\hell)}=\xi$ for some $\xi \in \{-,+\}$. Then $\xi=\sigma(k,m,l)=\overline{\sigma(i,m,j)}$. If $\sigma(k,m,j)=\xi$, we get a forbidden 4-tuple $(i,k,m,j)$, otherwise $\sigma(k,m,j)=\overline{\xi}$ and we get a forbidden 4-tuple $(k,m,j,\hell)$.
Finally, assume that $\sigma(i,k,j)=\sigma(k,j,\hell)=\xi$ for some $\xi \in \{-,+\}$.
The proof in this case is identical to the proof of the case $\hell<j$ and $\sigma(i,k,j)=\overline{\sigma(k,j,\hell)}=\xi$ in the previous paragraph. %This finishes the proof of the theorem.
\end{proof}

%-------------------------------------------------------------------

\subsection{Pseudolinear $x$-monotone drawings}

A drawing $D$ of a complete graph $K_n$ is {\em pseudolinear\/} (also {\em pseudogeometric} or {\em extendable}) if the edges of $D$ can be extended to unbounded simple curves that cross each other exactly once, thus forming an {\em arrangement of pseudolines}. The vertices of $D$ together with the  ${n \choose 2}$ pseudolines extending the edges are said to form a {\em pseudoarrangement of points} (also {\em generalized configuration of points}). Note that the pseudoarrangement of points extending $D$ is usually not unique as there is a certain freedom in choosing where the pseudolines extending disjoint noncrossing edges of $D$ cross.

It is well known that every arrangement of pseudolines can be made $x$-monotone by a suitable isotopy of the plane (this follows, for example, by the duality transform established by Goodman\mycite{G80_proof,GP84_semis}). Therefore, every pseudolinear drawing of $K_n$ is isotopic to an $x$-monotone pseudolinear drawing. Every rectilinear drawing of $K_n$ is $x$-monotone and pseudolinear, but there are pseudolinear drawings of $K_n$ that cannot be ``stretched'' to rectilinear drawings.

We show that $x$-monotone pseudolinear drawings of $K_n$ can be characterized in a combinatorial way by forbidden 4-tuples in the corresponding signature function, by further restricting the conditions on the signatures in Theorem~\ref{theorem_classif_semisimple}. In fact, the conditions in Theorem~\ref{theorem_classif_pseudolin} are precisely the {\em geometric constraints\/} that Peters and Szekeres\mycite{SP06_computer_17} used to restrict the set of signature functions in their investigation of the Erd\H{os}--Szekeres problem.
Figure~\ref{fig_semisimple} illustrates the classification of $4$-tuples from Theorem~\ref{theorem_classif_semisimple} and Theorem~\ref{theorem_classif_pseudolin}.

\begin{figure}
 \begin{center}
   \includegraphics[scale=1]{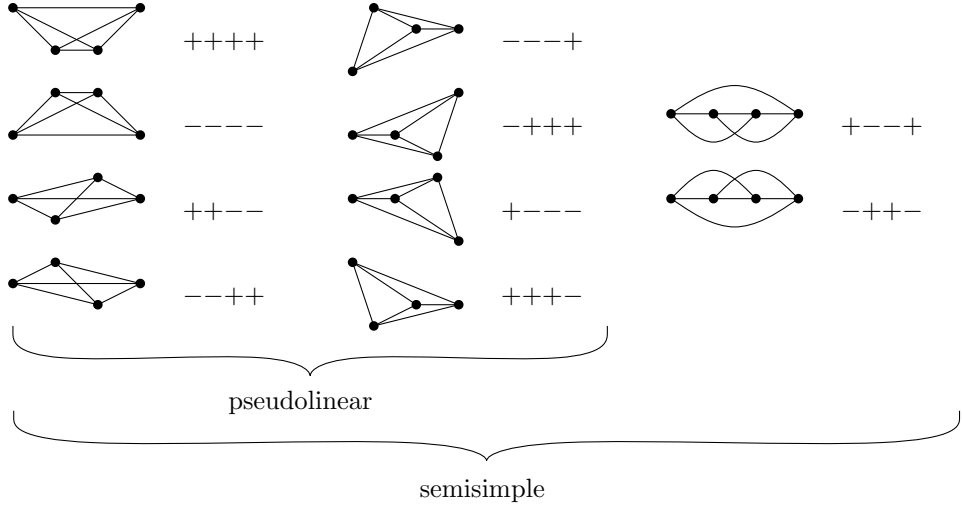}
  \caption{The $4$-tuples in pseudolinear and semisimple drawings.}
 \label{fig_semisimple}
 \end{center}
\end{figure}

\begin{theorem}\label{theorem_classif_pseudolin}
A signature function $\sigma \in \Sigma_n$ can be realized by a pseudolinear $x$-monotone drawing if and only if every ordered $4$-tuple of indices from $[n]$ is of one of the forms
\begin{align*}
& {+}{+}{+}{+},{+}{+}{+}{-},{+}{+}{-}{-},{+}{-}{-}{-}, \\
& {-}{-}{-}{-},{-}{-}{-}{+},{-}{-}{+}{+},{-}{+}{+}{+}
\end{align*}
in $\sigma$.
\end{theorem}

Pseudolinear drawings of complete graphs are equivalent to {\em CC systems\/} introduced by Knuth\mycommal\mycite{knuth92_axioms}\mycommar although this equivalence is not easily seen. The CC systems are ternary {\em counter-clockwise relations\/} of finite sets satisfying a certain set of five axioms involving triples, 4-tuples or 5-tuples of elements. CC systems generalize the {\em order types\/} of planar point sets in general position: an ordered triple in the counter-clockwise relation is interpreted as a triple of points in the plane placed in the counter-clockwise order, like a triple with signature $+$ in the signature function.  Unlike the signature functions, the CC systems have no fixed ordering of the elements. Therefore, some of the axioms for CC systems involve 5-tuples of elements, whereas 4-tuples are sufficient in the case of signature functions. In fact, the axioms of CC systems specify exactly that every 5-tuple of elements can be realized as a point set in the plane. Knuth\mycite{knuth92_axioms} established a correspondence between CC systems
and {\em reflection networks\/} (also called {\em wiring diagrams}), which are simple arrangements of pseudolines dual to the pseudoarrangements of points extending the pseudolinear drawings of complete graphs. Knuth\mycite{knuth92_axioms} also showed a two-to-one correspondence between CC systems and {\em uniform acyclic oriented matroids of rank 3\/} on the same underlying set. Here the CC system is, in fact, the {\em chirotope\/} of the corresponding oriented matroid.

%In an extended abstract of the 13th SoCG conference, 
Streinu\mycite{Str97_clusters} characterized sets of signed circular permutations ({\em directed clusters of stars}) that arise from generalized configurations of $n$ points as circular sequences of pseudolines at each of the $n$ points, and provided an $O(n^2)$ drawing algorithm, partially similar to ours. It is easy to show that the set of signed circular permutations determines the orientation of all triangles (and thus the corresponding CC system) and vice versa. However, many details are omitted in the extended abstract\mydotl\mycite{Str97_clusters}\mydotr

Felsner and Weil\mycite{Fel04_book,FW01_sweeps} proved that {\em triangle-sign functions\/} of simple arrangements of $n$ pseudolines are precisely those functions $f:{[n]\choose 3}\rightarrow \{+,-\}$ that are monotone on all $4$-tuples. This is the same condition as the condition on signature functions in Theorem~\ref{theorem_classif_pseudolin}. That is, Theorem~\ref{theorem_classif_pseudolin} is a dual analogue of Felsner's and Weil's result. Felsner and Weil\mycite{Fel04_book,FW01_sweeps} also introduced {\em $r$-signotopes}, a notion unifying permutations, allowable sequences and monotone triangle-sign functions of simple arrangements. In this notation, the signature functions satisfying the conditions of Theorem~\ref{theorem_classif_pseudolin} are 3-signotopes.

%by more or less painful process
Although Theorem~\ref{theorem_classif_pseudolin} can be deduced from any of these previous results, we still believe that providing a direct, self-contained proof has its merit.

\subsubsection{Proof of Theorem~\ref{theorem_classif_pseudolin}}
%\begin{proof}[Proof of Theorem~\ref{theorem_classif_pseudolin}]
Clearly, every pseudolinear $x$-monotone drawing of $K_4$ is isotopic to one of the eight drawings of $K_4$ in the first two columns in Figure~\ref{fig_semisimple}, and thus its signature function has one of the corresponding eight forms. 

Let $\sigma$ be a signature function such that every 4-tuple is of one of the eight allowed forms in $\sigma$. We show that there is a pseudolinear $x$-monotone drawing $D$ of $K_n$ which induces $\sigma$. Unlike in the proof of Theorem~\ref{theorem_classif_semisimple}, we do not provide an explicit construction of the drawing. However, our proof can be easily transformed into a polynomial algorithm finding such a drawing.

Again, we use the points $v_i=(i,0)$, $i \in [n]$, as vertices. 
For $m \in [n]$, let $L_m$ be the vertical line containing $v_m$. Let $L_0$ be the vertical line containing the point $(0,0)$.
% and let $L_{n+1}$ be the vertical line containing the point $({n+1},0)$.

To determine the drawing of $K_n$ and the pseudolines extending the edges, up to a combinatorial equivalence, it suffices to specify the left and right vertical orders of the pseudolines crossing at each of the points $v_i$, and the relative positions of the
intersections of the pseudolines with the lines $L_0, L_1, \dots, L_n$. %L_{n+1}$.

For $i,j \in [n], i\neq j,$ let $p_{i,j}$ be the pseudoline extending the edge $v_iv_j$. We emphasize that we use both $p_{i,j}$ and $p_{j,i}$ to denote the same pseudoline.
We draw the pseudolines in the following way; see Figure~\ref{fig_kresleni_pseudoprimek}. For every $i\in [n]$, we draw a portion of each pseudoline $p_{i,j}$ containing $v_i$ as two short segments joining points $v_i-(\varepsilon,\delta_j)$, $v_i$ and $v_i+(\varepsilon,\delta'_j)$, where $\varepsilon,\delta_j$ and $\delta_j'$ are sufficiently small and the relative order of the $y$-coordinates $\delta_j$ ($\delta'_j$) is consistent with the left (right, respectively) vertical order of the pseudolines at $v_i$. It will follow from the construction that we can take $\delta'_j=-\delta_j$, so the two segments actually form one segment with midpoint $v_i$. We also choose the intersection points of the pseudolines $p_{i,j}$ with the lines $L_m$ $(i,j \neq m)$ sufficiently far from the points $v_m$ and consistently with the relative positions specified. Then for each pseudoline $p_{i,j}$, we connect consecutive intersections with lines $L_m$, $m\in \{0,1,\dots, n\} \setminus\{i,j\}$, and points $v_i\pm(\varepsilon,\delta_j)$ and $v_j\pm(\varepsilon,\delta_i)$ by straight-line segments. Finally, we attach horizontal rays starting at intersections of $p_{i,j}$ with $L_0$ directed to the left, and similarly, horizontal rays starting at intersections of $p_{i,j}$ with $L_{n}$ (if $i,j \neq n$) and at points $v_n+(\varepsilon,\delta'_j)$, directed to the right.

%%For $k\in \{0,1,\dots,n\}, k\neq i,j$, let $x_{i,j,k}$ be the intersection of $p_{i,j}$ with $L_k$.

\begin{figure}
 \begin{center}
   \includegraphics[scale=1]{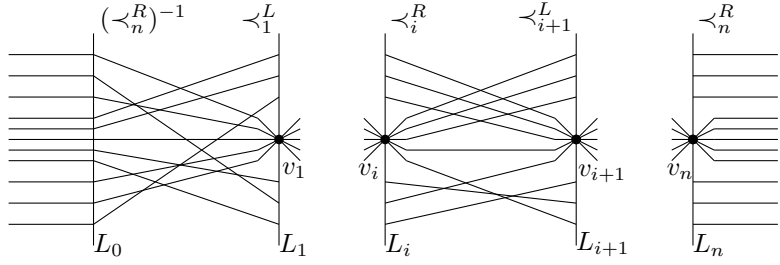}
  \caption{Drawing the curves $p_{i,j}$.}
 \label{fig_kresleni_pseudoprimek}
 \end{center}
\end{figure}

\begin{figure}
 \begin{center}
   \includegraphics[scale=1]{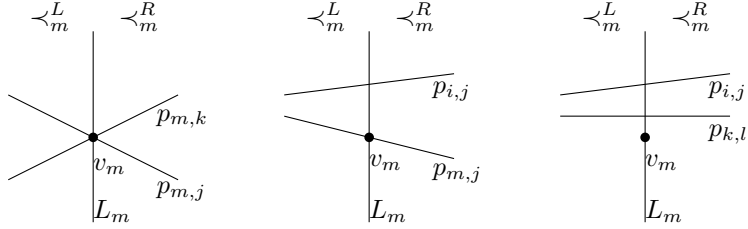}
  \caption{Illustration of the binding conditions for the orders $\prec_m^L$ and $\prec_m^R$.}
 \label{fig_def_usp_u_vrcholu}
 \end{center}
\end{figure}

We represent the order of intersections of the pseudolines with a vertical line by the order of the corresponding pseudolines. 
For $m\in [n]$, we define two total orders $\prec_m^L$ and $\prec_m^R$ on the set 
%$\mathcal{P}$ 
$\mathcal{P}=\{p_{i,j}; 1\le i<j\le n\}$ of all pseudolines. The order $\prec_m^L$ ($\prec_m^R$) represents the vertical order of the pseudolines in the left (right, respectively) neighborhood of $L_m$. See Figure~\ref{fig_def_usp_u_vrcholu}. We require the two orders $\prec_m^L$ and $\prec_m^R$ to be mutually inverse on the set $\mathcal{P}_m=\{p_{i,j}; 1\le i<j\le n, m\in\{i,j\}\}$ of pseudolines containing $v_m$ and identical otherwise, that is, 

\begin{itemize}

\item $p_{m,k}\prec_m^L p_{m,j}$ if and only if $p_{m,j}\prec_m^R p_{m,k}$, for all $j,k \in [n]$ such that $j,k,m$ are distinct, and
\item $p_{i,j}\prec_m^L p_{k,l}$ if and only if $p_{i,j}\prec_m^R p_{k,l}$, for all $i,j,k,l \in [n]$ such that $i<j$, $k<l$, $\{i,j\}\neq\{k,l\}$, and $\{i,j\}\cap\{k,l\}\cap\{m\}=\emptyset$.

\end{itemize}

We also define a total order $\prec_0$ on $\mathcal{P}$ as follows:
\begin{itemize}

\item  $\prec_0\equiv(\prec_n^R)^{-1}$.

\end{itemize}

That is, 
$\prec_0$ is the inverse of $\prec_n^R$. The objective here is to 
make every two pseudolines cross an odd number of times (in particular, at least once).

Further conditions on the orders $\prec_m^L$ and $\prec_m^R$ are determined by the signature function $\sigma$; see Figure~\ref{fig_podminky_nad_pod_orders}. %as follows
For $i \in [n]$, we fix the orders $\prec_i^L$ and $\prec_i^R$ on $\mathcal{P}_i$ as total orders in the following way. For all $j,k\in [n]$ such that $i\neq j<k\neq i$,
\begin{itemize}

\item if $i<j<k$, then $p_{i,j}\prec_i^R p_{i,k}$ if $\sigma(i,j,k)=+$ and $p_{i,k}\prec_i^R p_{i,j}$ if $\sigma(i,j,k)=-$,
\item if $j<k<i$, then $p_{i,j}\prec_i^R p_{i,k}$ if $\sigma(j,k,i)=+$ and $p_{i,k}\prec_i^R p_{i,j}$ if $\sigma(j,k,i)=-$, %and
\item if $j<i<k$, then $p_{i,j}\prec_i^R p_{i,k}$ if $\sigma(j,i,k)=+$ and $p_{i,k}\prec_i^R p_{i,j}$ if $\sigma(j,i,k)=-$.

\end{itemize}

\begin{figure}
 \begin{center}
   \includegraphics[scale=1]{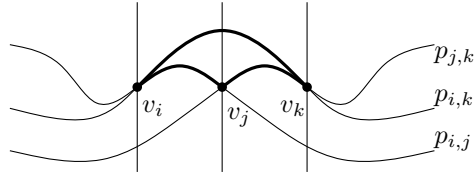}
  \caption{Three pseudolines determined by $v_i,v_j$ and $v_k$.}
 \label{fig_podminky_nad_pod_orders}
 \end{center}
\end{figure}

To show that $\prec_i^R$ and $\prec_i^L$ are total orders on $\mathcal{P}_i$, we need to verify the transitivity of $\prec_i^R$. Suppose for contrary that $p_{i,j}\prec_i^R p_{i,k} \prec_i^R p_{i,l} \prec_i^R p_{i,j}$ for some $j,k,l \in [n]\setminus\{i\}$. We may assume that $j<k,l$. 
In the following table we list the eight cases of $\sigma(\{i,j,k,l\})$ according to the relative order of $i,j,k,l$. The symbol $\xi$ stands for a sign that is not determined.
%If $i<j<k<l$ then the 4-tuple $(i,j,k,l)$ is of the form ${+}{-}{+}\xi$ in $\sigma$ for some $\xi\in\{+,-\}$. If $i<j<l<k$ then $(i,j,l,k)$ is of the form ${-}{+}{-}\xi$.

\begin{center}
\begin{tabular}{cc|cc}
order & $\sigma(\{i,j,k,l\})$ & order & $\sigma(\{i,j,k,l\})$\\
\hline
$i<j<k<l$ & ${+}{-}{+}\xi$  &  $i<j<l<k$ & ${-}{+}{-}\xi$\\
$j<k<l<i$ & $\xi{+}{-}{+}$  &  $j<l<k<i$ & $\xi{-}{+}{-}$\\
$j<i<k<l$ & ${+}{-}\xi{+}$  &  $j<i<l<k$ & ${-}{+}\xi{-}$\\
$j<k<i<l$ & ${+}\xi{-}{+}$  &  $j<l<i<k$ & ${-}\xi{+}{-}$\\
\end{tabular}
\end{center}

It follows that in every relative ordering, the indices $i,j,k,l$ form a forbidden 4-tuple. Therefore, both $\prec_i^R$ and $\prec_i^L$ are transitive on $\mathcal{P}_i$.

For every $i,j,k \in [n]$ such that $i\neq j<k\neq i$ and for every $p \in \mathcal{P}_i$, we also fix the following conditions:
\begin{itemize}
  \item if $i<j<k$, then $p\prec_i^R p_{j,k}$ if $\sigma(i,j,k)=-$ and $p_{j,k} \prec_i^R p$ if $\sigma(i,j,k)=+$,
  \item if $j<k<i$, then $p\prec_i^R p_{j,k}$ if $\sigma(j,k,i)=-$ and $p_{j,k} \prec_i^R p$ if $\sigma(j,k,i)=+$,
  \item if $j<i<k$, then $p\prec_i^R p_{j,k}$ if $\sigma(j,i,k)=+$ and $p_{j,k} \prec_i^R p$ if $\sigma(j,i,k)=-$.
\end{itemize}
These conditions represent the above/below relations of the pseudolines $p_{j,k}$ and the points $v_i$ implied by $\sigma$ (see Figure~\ref{fig:sigma}).

It is easy to see that all the conditions required so far for the orders $\prec_m^L, \prec_m^R$ and $\prec_0$ can be simultaneously satisfied. 
For example, for crossings of $L_m$, $m\in [n]$, with the pseudolines disjoint with $v_m$, the conditions only 
specify a partition of these pseudolines into two subsets: those crossing $L_m$ below $v_m$ and those crossing $L_m$ above $v_m$.

Finally, we choose total orders $\prec_m^L, \prec_m^R$ and $\prec_0$ on $\mathcal{P}$ satisfying all the required conditions and such that the total number of crossings of the pseudolines is minimized. Combinatorially, this last condition is equivalent to minimizing the total number of inversions between pairs of permutations corresponding to $\prec_i^R$ and $\prec_{i+1}^L$, for all $i\in [n-1]$, and the pair of permutations corresponding to $(\prec_n^R)^{-1}$ and $\prec_1^L$.

Let $A$ be an arrangement of piecewise linear curves $p_{i,j}$ constructed from the total orders $\prec_m^L, \prec_m^R$ and $\prec_0$. %up to the choice of the precise positions of the intersections of the curves with the lines $L_m$. 
Assume that no three curves from $A$ cross at the same point, except for the points $v_1, v_2, \dots, v_n$.
We show that every two curves in $A$ cross exactly once and thus deserve to be referred to as pseudolines.

Let $e,f$ be two $x$-monotone curves from the arrangement $A$. A {\em bigon\/} $B$ {\em formed by $e$ and $f$} is a closed topological disc bounded by two simple arcs $e', f'$ that have common endpoints and disjoint relative interiors, and such that $e'$ is a portion of $e$ and $f'$ is a portion of $f$. The common endpoints of $e'$ and $f'$ are the {\em vertices\/} of $B$. 
%%Note that each of the arcs $e'$, $f'$ forms a top or the bottom part of the boundary of $B$. 

It will be convenient to consider $A$ as an arrangement of curves on the M\"obius strip obtained from the infinite rectangle $\{(x,y)\in \mathbb{R}^2; 0\le x\le n+\varepsilon, y\in \mathbb{R}\}$ by identifying each point $(0,y)$, $y\in \mathbb{R}$, with the point $(n+\varepsilon,-y)$. 
%gluing $L_0$ with the vertical line $L_{n+\varepsilon}$ containing the point $(n+\varepsilon,0}$ in the opposite directions.
We extend the notion of a bigon to include also {\em special bigons\/} that are bounded by portions of two curves from $A$ in the M\"obius strip and intersect the line $L_0$. A bigon that does not intersect $L_0$ is an {\em ordinary bigon}.
Observe that if two curves $e$ and $f$ cross $k$ times, then $e$ and $f$ form exactly $k-1$ ordinary bigons and one special bigon.

A bigon $B$ is {\em empty\/} if $B \cap \{v_1, v_2, \dots, v_n\}=\emptyset$. 
A bigon $B$ is {\em smooth\/} if the boundary of $B$ does not intersect $\{v_1, v_2, \dots, v_n\}$. See Figure~\ref{fig_bigony}.

\begin{figure}
 \begin{center}
   \includegraphics[scale=1]{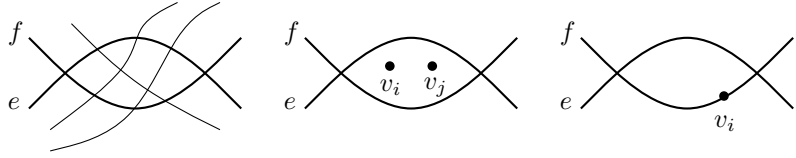}
  \caption{Bigons formed by curves $e,f$. Left: a minimal empty bigon. Middle: a smooth bigon. Right: a bigon that is neither smooth nor empty.}
 \label{fig_bigony}
 \end{center}
\end{figure}

\begin{claim}\label{claim_empty_bigon}
No two curves from $A$ form an empty bigon.
\end{claim}

\begin{proof}
Suppose that $e$ and $f$ are two curves from $A$ that form an empty bigon $B$. Let $e' \subset e$ and $f' \subset f$ be the two arcs forming the boundary of $B$.
%, with a portion of $e$ forming the bottom boundary of $B$. 
Suppose further that $B$ is inclusion minimal, among all pairs of pseudolines. Moreover, we may suppose that both $e'$ and $f'$ are inclusion minimal among all arcs forming the bottom or the top boundary of some bigon. Then every curve $g$ from $A$ distinct from $e$ and $f$ is either disjoint with $B$ or crosses both $e'$ and $f'$ exactly once.
%the same number of times (once or twice), 
%since every connected component of $g\cap B$ is a curve with one endpoint on $e'$ and the other endpoint on $f'$. 
We can thus redraw $e$ along $f'$ outside $B$ and decrease the total number of crossings by two. After this operation, the resulting arrangement still satisfies all the conditions specified by the orders $\prec_m^L, \prec_m^R$ and $\prec_0$, as the neighborhoods of the points $v_i$ and all the above/below relations of the pseudolines and points $v_i$ remain unaffected.
\end{proof}

\begin{corollary}\label{cor_smooth_bigon}
Every smooth bigon formed by two curves from $A$ contains at least one point $v_m$ in its interior.
\end{corollary}

For $i,j\in[n]$, $i\neq j$, let $e_{i,j}$ be the portion of the curve $p_{i,j}$ between the points $v_i$ and $v_j$, representing the edge $v_iv_j$ of $K_n$.

\begin{claim}\label{claim_sousedni}
Every two curves from $A$ sharing a point $v_i$ cross only at $v_i$.
\end{claim}

\begin{proof}
Suppose that for some $j,k \in [n]\setminus \{i\}$, $j<k$, the curves $p_{i,j}$ and $p_{i,k}$ cross more than once.
By symmetry, we may assume that $i<k$ and $p_{i,j} \prec_i^R p_{i,k}$. If $j<i$, we may further assume that $e_{i,j}$ crosses $p_{i,k}$ at most as many times as $e_{i,k}$ crosses $p_{i,j}$.
We have five cases; see Figure~\ref{fig_pij_pik_cases}.

\begin{enumerate}[(i)]

\item $i<j$ and $e_{i,j}$ crosses $e_{i,k}$. By the definition of $\prec_i^R$, we have $\sigma(i,j,k)=+$. Consequently, $p_{i,k}$ crosses $L_j$ above $v_j$. This further implies that $e_{i,j}$ and $e_{i,k}$ cross at least twice and thus they form a smooth ordinary bigon $B$. Let $v_m$ be a point in the interior of $B$ guaranteed by Corollary~\ref{cor_smooth_bigon}. We have $i<m<j<k$, $\sigma(i,m,j)=+$ and $\sigma(i,m,k)=-$, which implies that $\sigma(i,m,j,k)={+}{-}{+}{\xi}$ for some $\xi \in \{-,+\}$.

%%not happening \item $i<j$ and $e_{i,j}$ crosses $p_{i,k}$ but not $e_{i,k}$, 

\item $i<j$ and $e_{i,k}$ crosses $p_{i,j}$ but not $e_{i,j}$. Again, we have $\sigma(i,j,k)=+$. Consequently, $p_{i,k}$ crosses $L_j$ above $v_j$ and $p_{i,j}$ crosses $L_k$ below $v_k$. This further implies that $p_{i,j}$ and $e_{i,k}$ cross at least twice (not counting the point $v_i$) and thus they form a smooth ordinary bigon with a point $v_m$ in its interior. Taking the leftmost such bigon, we have $i<j<m<k$ and $\sigma(i,j,m,k)={-}{+}{-}{\xi}$.

\item $i<j$, $e_{i,j}$ does not cross $p_{i,k}$ and $e_{i,k}$ does not cross $p_{i,j}$. In this case all three points $v_i$, $v_j$ and $v_k$ are on the boundary of the same bigon $B$ formed by $p_{i,j}$ and $p_{i,k}$ and only $v_i$ is a vertex of $B$. Since $p_{i,j}$ and $p_{i,k}$ cross at least three times, they form at least three bigons. At most two of the bigons contain $v_i$, thus at least one of them, $B'$, is smooth and has a point $v_m$ in its interior. We may assume that $B'$ shares a vertex with $B$. Note that either of $B$ and $B'$ can be special, so we have two cases: $m>k$ or $m<i$. In the first case we have $\sigma(i,j,k,m)={+}{-}{+}{\xi}$, in the second case we have $\sigma(m,i,j,k)={+}{-}{\xi}{+}$.

\item $j<i$ and $e_{i,k}$ crosses $p_{i,j}$. Since $v_k$ lies above $p_{i,j}$, the curves $e_{i,k}$ and $p_{i,j}$ cross at least twice and thus form a smooth bigon, containing a point $v_m$ in its interior. We have $j<i<m<k$ and $\sigma(j,i,m,k)={-}{+}{\xi}{-}$.

\item $j<i$, $e_{i,j}$ does not cross $p_{i,k}$ and $e_{i,k}$ does not cross $p_{i,j}$. The curves $p_{i,j}$ and $p_{i,k}$ form at least three bigons, but only two of them, $B_1$ and $B_2$, contain $v_i$. Let $B$ be the bigon other than $B_1$ sharing a vertex with $B_2$. By the assumptions, $B\cap \{v_i,v_j,v_k\}=\emptyset$, so $B$ is smooth and contains some point $v_m$. We have either $m<j<i<k$ with $\sigma(m,j,i,k)={+}{\xi}{-}{+}$, or $j<i<k<m$ with $\sigma(j,i,k,m)={+}{-}{\xi}{+}$.
\end{enumerate}

\begin{figure}
 \begin{center}
   \includegraphics[scale=1]{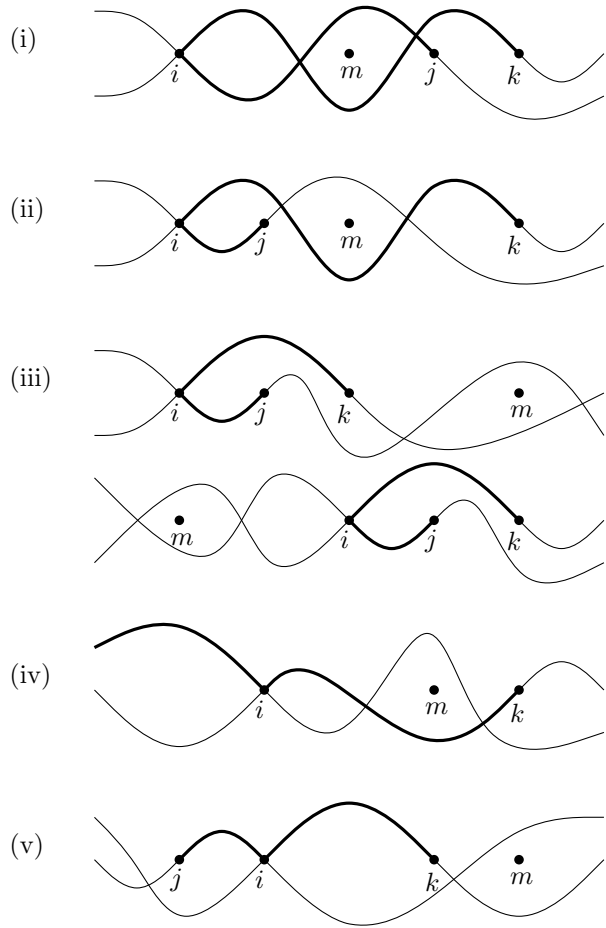}
  \caption{Bigons with vertex at $v_i$ are forbidden.}
 \label{fig_pij_pik_cases}
 \end{center}
\end{figure}

In every case there is a forbidden $4$-tuple, which is a contradiction.
\end{proof}

\begin{claim}\label{claim_nezavisle}
Let $i,j,k,l\in[n]$ such that $|\{i,j,k,l\}|=4$. Then $p_{i,j}$ and $p_{k,l}$ do not form a smooth bigon.
\end{claim}

\begin{proof} 
Let $B$ be a smooth bigon, containing a point $v_m$ in its interior. By Claim~\ref{claim_sousedni}, we may asume that $p_{i,m}$ and $p_{i,j}$ cross at most once. Therefore, $p_{i,m}$ enters and exits the bigon $B$ through $p_{k,l}$, perhaps more than once; see Figure~\ref{fig_pij_pkl_smooth}. Similarly, $p_{k,m}$ enters and exits the bigon $B$ through $p_{i,j}$. Since $p_{i,m}$ and $p_{k,m}$ cross at $v_m$ and they enter and exit $B$ through opposite sides, they cross at least twice. This is in contradiction with Claim~\ref{claim_sousedni}.
\end{proof}

\begin{figure}
 \begin{center}
   \includegraphics[scale=1]{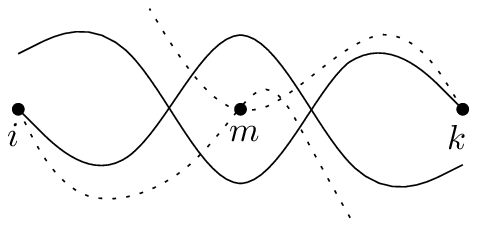}
  \caption{Smooth bigons are also forbidden.}
 \label{fig_pij_pkl_smooth}
 \end{center}
\end{figure}

If two curves from $A$ cross more than once, then by Corollary~\ref{cor_smooth_bigon}, Claim~\ref{claim_sousedni} and Claim~\ref{claim_nezavisle}, they are of the form $p_{i,j}$ and $p_{k,l}$ with $i,j,k,l$ distinct, and they form no smooth bigon. Therefore, every bigon formed by $p_{i,j}$ and $p_{k,l}$ contains at least one of the points $v_i,v_j,v_k,v_l$ on its boundary, but not at the vertices. In particular, $p_{i,j}$ and $p_{k,l}$ form exactly three bigons. Suppose that $v_i$ and $v_k$ are on the boundary of two different bigons. Up to symmetry, we are in one of the cases depicted in Figure~\ref{fig_pij_pkl_nonsmooth} (considering all symmetries of the annulus, there is just one case). The curve $p_{i,k}$ has to cross $p_{i,j}$ or $p_{k,l}$ at least twice, which contradicts Claim~\ref{claim_sousedni}. Similar argument for pairs $v_i,v_l$ and $v_j,v_k$ implies that all four points $v_i,v_j,v_k,v_l$ are on the boundary of the same bigon. Therefore at least two bigons are smooth, which contradicts Claim~\ref{claim_nezavisle}. This finishes the proof of the theorem.

\begin{figure}
 \begin{center}
   \includegraphics[scale=1]{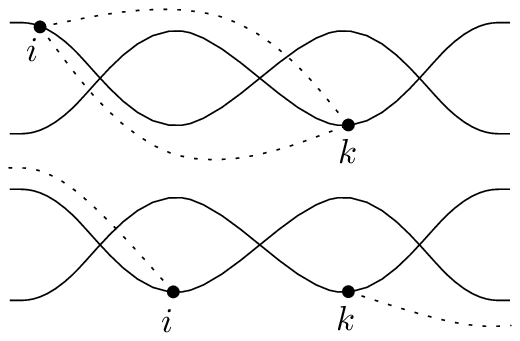}
  \caption{The curve $p_{ik}$ has to cross $p_{i,j}$ or $p_{k,l}$ in some other point than $v_i$ or $v_k$.}
 \label{fig_pij_pkl_nonsmooth}
 \end{center}
\end{figure}

%----------------------------------------------------------------------------------------------
\subsection{A remark on rectilinear drawings}

A similar characterization of {\em rectilinear\/} drawings of $K_n$ (equivalently, order types of planar point sets in general position) in terms of signature functions or CC systems with a finite number of forbidden configurations is impossible: for example, Bokowski and Sturmfels\mycite{BS89_infinite} constructed infinitely many minimal CC systems (simplicial affine 3-chirotopes) that are not realizable as sets of points in the plane. This and related results were also referred to by the phrase ``missing axiom for chirotopes is lost forever''.

Moreover, recognizing signature functions of rectilinear drawings of $K_n$ (or, order types of planar point sets in general position), is polynomially equivalent to rectilinear realizability of complete abstract topological graphs and to stretchability of pseudoline arrangements\mycommal\mycite{K11_simple}\mycommar which is polynomially equivalent to the existential theory of the reals\mydotl\mycite{Mnev88_universality}\mydotr In the terminology introduced by Schaefer\mycommal\mycite{Schaefer09_complexity}\mycommar these problems are $\exists\mathbb{R}$-complete. It is known that $\exists\mathbb{R}$-complete problems are in PSPACE\mycite{C88_some} and NP-hard, but they are not known to be in NP.

%-----------------------------------------------------------------------------

\subsection{Crossing minimal $x$-monotone drawings}

Note that in a simple $x$-monotone drawing of $K_n$, the crossings appear only between edges whose endpoints induce a 4-tuple of one of the forms ${+}{+}{+}{+}$, ${-}{-}{-}{-}$, ${+}{+}{-}{-}$, ${-}{-}{+}{+}$, ${-}{+}{+}{-}$, ${+}{-}{-}{+}$. Analogously as for the rectilinear drawings of $K_n$, we may call these 4-tuples {\em convex}. Then, for a simple $x$-monotone drawing $D$ of $K_n$ the crossing number of $D$ equals the number of convex 4-tuples.
%
%If we refer to these forms as {\em convex}, then there is a similar relation between crossings and convex 4-tuples as in the case of the rectilinear crossing number where $\mathrm{\overline{cr}}(K_n)$ corresponds to the number of convex 4-gons in the optimal rectilinear drawing of $K_n$. 
A similar notion of convexity for general $k$-tuples was used by Peters and Szekeres\mydotl\mycite{SP06_computer_17}\mydotr 

This description of crossings is convenient for computer calculations. Using it, we have obtained a complete list of optimal $x$-monotone drawings of $K_n$ for $n\le 10$. To enumerate ``essentially different'' drawings we used the following approach.

Let $D$ be an $x$-monotone drawing of $K_n$ which induces a signature function $\sigma$. We can assume that the vertices are points placed on the same horizontal line (the $x$-axis).
The following operations on $D$ and $\sigma$ produce a signature function $\sigma'$ of a simple monotone drawing $D'$ that is homeomorphic to $D$ on the sphere, by a homeomorphism that does not necessarily preserve the labels of vertices. In some cases we just describe the transformation of the drawing; the new signature function $\sigma'$ can be then computed in a straightforward way.

\begin{enumerate}[(a)]
\item {\em Vertical reflection}: setting $\sigma'(i,j,k)=\overline{\sigma(i,j,k)}$ for every $(i,j,k) \in T_n$.

\item {\em Horizontal reflection}: setting $\sigma'(i,j,k)=\sigma(n+1-k,n+1-j,n+1-i)$ for every $(i,j,k) \in T_n$.

\item {\em Shifting $v_1$}: if every edge incident to $v_1$ lies completely above or completely below the $x$-axis, that is, $\sigma(1,i,k)=\sigma(1,j,k)$ for every $k \in \{3,\ldots,n\}$ and $1 < i,j < k$, then we can move $v_1$ to the position of $v_n$ and move every $v_{i+1}$ to the position of $v_i$, for every $1 \le i \le n-1$.

\item {\em Switching consecutive points}: let $j \in [n-1]$. If there is a $\xi \in \{-,+\}$ such that $\sigma(j,j+1,k) = \xi$ for every $j+1<k\le n$ and $\sigma(i,j,j+1)=\overline{\xi}$ for every $1 \le i < j$, then we can switch the positions of $v_j$ and $v_{j+1}$. After the switch, we have $\sigma'(j,j+1,k) = \overline{\xi}$ for every $j+1<k\le n$ and $\sigma'(i,j,j+1)=\xi$ for every $1 \le i < j$. %move $v_{i+1}$ to a position of $v_i$.

\item {\em Redrawing the edge $v_1v_n$}: in every crossing minimal $x$-monotone drawing, the edge $v_1v_n$ crosses no other edge, since we can always redraw this edge along the top or the bottom part of the boundary of the outer face. The signature function $\sigma$ thus satisfies $\sigma(1,i,n)=\xi$ for some $\xi\in\{{+},{-}\}$ and for every $i$, $1<i<n$. We may thus simultaneously change all the signatures $\sigma(1,i,n)$.

\end{enumerate}

We say that two $x$-monotone drawings $D$ and $D'$ are {\em switching equivalent}
%, written $D \sim D'$, 
if there is a sequence of operations (a)--(e) such that, when applied to $D$, we obtain a drawing which has the same signature function as $D'$. We have found representatives of all switching equivalence classes of crossing minimal $x$-monotone drawings of $K_n$, for $n \le 10$. Their numbers are given in Table~\ref{tab:nonequivalent}.

\'{A}brego {\it et al.}\mycite{abre12_2page} proved that for every even $n$, there is a unique crossing minimal $2$-page book drawing of $K_n$, up to a homeomorphism of the sphere. We have found crossing minimal $x$-monotone drawings of $K_8$ and $K_{10}$ that are not homeomorphic to $2$-page book drawings. There are exactly two such drawings of $K_8$; see Figure~\ref{fig_minimal_no2page}.
We do not have a construction of such drawings of $K_n$ for arbitrarily large $n$.

\begin{figure}
 \begin{center}
   \includegraphics[scale=0.8]{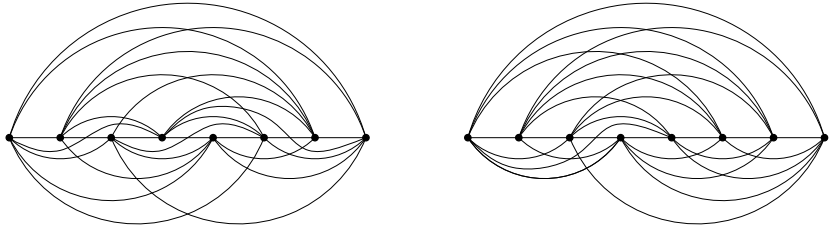}
  \caption{Left: a crossing minimal $x$-monotone drawing of $K_8$ homeomorphic to the cylindrical drawing. Right: a crossing minimal $x$-monotone drawing of $K_8$ that is not homeomorphic to a $2$-page book drawing and neither to the cylindrical drawing.}
 \label{fig_minimal_no2page}
 \end{center}
\end{figure}

\begin{table}
\begin{center}
\begin{tabular}{ L{3.7cm} | C{0.7cm} | C{0.7cm} | C{0.7cm} | C{0.7cm} | C{0.7cm} | C{0.7cm}}
Number of vertices & 5 & 6 & 7 & 8 & 9 & 10 \\ \hline
Number of drawings & 1 & 1 & 5 & 3 & 510 & 38 \\
\end{tabular}
\end{center}
\caption{Numbers of switching equivalence classes of crossing minimal $x$-monotone drawings of $K_n$ for $n \le 10$.}
\label{tab:nonequivalent}
\end{table}

%==========================================================================================================
\section{Weakly semisimple and shellable drawings}\label{section_+-}

Our proof of Theorem~\ref{theorem_1} for semisimple monotone drawings, as well as the earlier proof by \'Abrego {\em et al.}\mycommal\mycite[Theorem 1.1]{abre13_more}\mycommar do not use all properties of monotone drawings. Both rely only on the fact that the vertices of the drawing can be ordered as $v_1, v_2, \dots, v_n$ so that for every pair $i,j$ with $1\le i < j \le n$, the vertices $v_i$ and $v_j$ are on the outer face of the drawing induced by the interval of vertices $v_i, v_{i+1}, \dots, v_j$. Pedro Ramos\mycite{ram13_egc} introduced the term {\em shellable drawings\/} for these drawings of $K_n$. 
\'Abrego {\em et al.}\mycite{aamrs13_shellable} later observed that a still more general condition, $s$-shellability for some $s\ge n/2$, is sufficient, since the depth of the recursion in the proof is only $n/2$. A drawing of a complete graph with a vertex set $V$ is called {\em $s$-shellable\/} if there is a subset of vertices $v_1, v_2, \dots, v_s \in V$ such that for every pair $i,j$ with $1\le i < j \le s$, the vertices $v_i$ and $v_j$ are on the outer face of the drawing induced by $V\setminus \{v_1, v_2, \dots, v_{i-1}, v_{j+1}, v_{j+2}, \dots, v_s\}$.
In our version of this definition, we require $v_1$ and $v_s$ to be incident with the outer face; this is slightly more restrictive compared to the original definition in\mydotl\mycite{aamrs13_shellable}\mydotr
Informally speaking, $s$-shellable drawings consist of two parts: the first part is a shellable drawing of $K_s$, the second part is an arbitrary drawing of the remaining vertices and edges that does not block the shelling of the first part. If $s\ge 3$, this means, in particular, that all vertices from the second part ``see'' the vertices in the first part in the same cyclic order. The class of $s$-shellable drawings includes, for example, all drawings with a crossing-free cycle of length $s$, with at least one edge of the cycle incident with the outer face\mydotl\mycite{aamrs13_shellable}\mydotr
Note that the notions \emph{shellable} and \emph{$n$-shellable} coindide for drawings of $K_n$.

Following this notation, we call the sequence $v_1, v_2, \dots, v_n$ from the definition of a shellable drawing of $K_n$ a {\em shelling sequence\/} of the drawing, which is similar to the term {\em $s$-shelling\/} introduced by \'Abrego {\em et al.}\mycite{aamrs13_shellable}.

\'Abrego {\em et al.}\mycite{aamrs13_shellable} also considered the class of $x$-bounded drawings, which form a subclass of shellable drawings and generalize $x$-monotone drawings. A drawing of a graph is {\em $x$-bounded\/} if no two vertices share the same $x$-coordinate and every interior point of every edge $uv$ lies in the interior of the strip bounded by two vertical lines passing through the vertices $u$ and $v$. Fulek {\em et al.}\mycite{FPSS13_ht_monotone} showed that every $x$-bounded drawing $D$ can be transformed into an $x$-monotone drawing $D'$, while keeping the rotation system and  the parity of the number of crossings of every pair of edges fixed. This implies, in particular, that $\text{ocr}(D)=\text{ocr}(D')$. Also $D'$ is weakly semisimple if and only if $D$ is weakly semisimple. Therefore, the lower bound from Theorem~\ref{theorem_1} extends to all weakly semisimple $x$-bounded drawings of $K_n$.

It is not a priori clear that shellable drawings are essentially different from monotone or $x$-bounded drawings, since the conditions for shellability and $x$-boundedness are very similar at first sight. In Subsection~\ref{sub_shellable_not_monotone} we show that simple shellable drawings are indeed more general than simple monotone drawings, but the difference is rather subtle. By a somewhat detailed analysis, which we do not include here, it can be shown that every simple shellable drawing of $K_n$ can be decomposed into three monotone drawings, in a very specific way. 

Apart from following the proof of Theorem~\ref{theorem_1}, we may obtain a lower bound on the crossing number of shellable drawings of $K_n$ by the following straightforward reduction to the monotone crossing number of $K_n$, using the combinatorial characterization of $x$-monotone drawings.
%There is a straightforward reduction from the crossing number of shellable drawings of $K_n$ to the monotone crossing number of $K_n$, using the combinatorial characterization of $x$-monotone drawings.

\begin{proposition}\label{prop_shellable_semisimple}
Let $D$ be a semisimple shellable drawing of $K_n$. There is a semisimple $x$-monotone drawing $D'$ of $K_n$ with $\text{ocr}(D')=\text{ocr}(D)$.
\end{proposition}

We note that the drawing $D'$ obtained in Proposition~\ref{prop_shellable_semisimple} does not necessarily preserve the parity of the number of crossings between a given pair of edges. Moreover, it is also possible that for a simple shellable drawing $D$, we obtain a monotone drawing $D'$ where some pair of edges cross more than once; see Figure~\ref{fig_shellable_monotone_K5}.

Let $v_1, v_2, \dots, v_n$ be the vertices of a semisimple drawing $D$ of $K_n$. The {\em order type\/} of $D$ is the function $\sigma: {[n]\choose 3} \rightarrow \{+,-\}$ defined in the following way: for $1\le i<j<k\le n$, $\sigma(i,j,k)=+$ if the triangle $v_iv_jv_k$ is drawn counter-clockwise and $\sigma(i,j,k)=-$ if the triangle $v_iv_jv_k$ is drawn clockwise. This generalizes the definition of the signature function for semisimple monotone drawings. As in the previous section, we use the shortcut $\sigma(i,j,k,l)$ for the sequence of four signs $\sigma(i,j,k)\sigma(i,j,l)\sigma(i,k,l)\sigma(j,k,l)$.

\begin{figure}
 \begin{center}
   \includegraphics[scale=1]{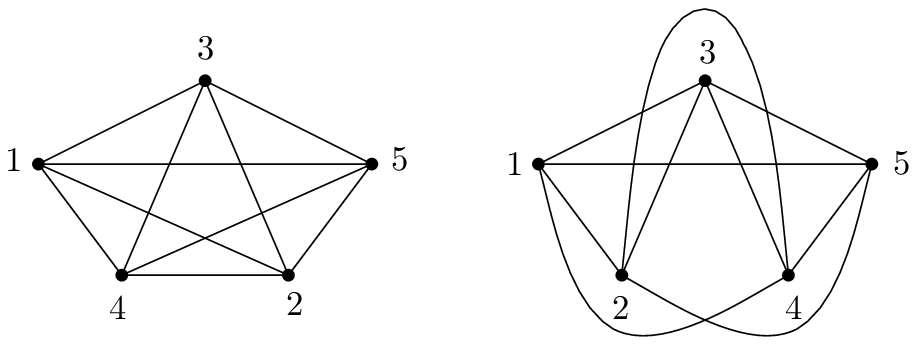}
  \caption{A simple shellable drawing (left) and the corresponding semisimple monotone drawing (right). Note that the left drawing is both shellable and monotone; however, its shelling sequence $1,2,3,4,5$ is not its monotone sequence.}
 \label{fig_shellable_monotone_K5}
 \end{center}
\end{figure}

\begin{proof}
Let $v_1, v_2, \dots, v_n$ be a shelling sequence of $D$. Let $\sigma$ be the order type of $D$.
We show that $\sigma$ satisfies the assumptions of Theorem~\ref{theorem_classif_semisimple}, and therefore can be realized by a semisimple monotone drawing. Let $v_i,v_j,v_k,v_l$ be a $4$-tuple of vertices with $1\le i<j<k<l\le n$. Then the drawing of $K_4$ induced by $v_i,v_j,v_k,v_l$ has $v_i$ and $v_l$ on its outer face. To verify the assumptions of Theorem~\ref{theorem_classif_semisimple}, it is sufficient to show that none of the cases $\sigma(i,j,k,l)={+}{-}{+}{\xi}$, $\sigma(i,j,k,l)={-}{+}{-}{\xi}$, $\sigma(i,j,k,l)={\xi}{+}{-}{+}$ or $\sigma(i,j,k,l)={\xi}{-}{+}{-}$, with $\xi\in\{+,-\}$, occurs. Suppose the contrary. Due to symmetry, we may suppose that $\sigma(i,j,k,l)={+}{-}{+}{\xi}$. This means that reading the linear counter-clockwise order of the edges incident with $v_i$ starting from the outer face, we encounter the edge $v_iv_j$ before the edge $v_iv_k$, $v_iv_k$ before $v_iv_l$, and $v_iv_l$ before $v_iv_j$; a contradiction.

Let $D'$ be a semisimple monotone drawing realizing $\sigma$. Every $4$-tuple of vertices in $D$ induces a drawing of $K_4$ with at most one pair of edges crossing oddly. This is clear if $D$ is simple; for semisimple drawings this is proved in the claim in the proof of Lemma~\ref{lemma2}. Call a $4$-tuple of vertices in $D$ or $D'$ {\em odd\/} if it induces exactly one pair of edges crossing oddly and {\em even\/} otherwise.
To finish the proof, it remains to show that odd (even) $4$-tuples of vertices in $D$ correspond to odd (even, respectively) $4$-tuples in $D'$.

Odd (also convex) $4$-tuples in $D'$ are of one of the forms ${+}{+}{+}{+}$, ${-}{-}{-}{-}$, ${+}{+}{-}{-}$, ${-}{-}{+}{+}$, ${-}{+}{+}{-}$, ${+}{-}{-}{+}$. Even $4$-tuples in $D'$ are of one of the forms ${+}{+}{+}{-}$, ${-}{+}{+}{+}$, ${-}{-}{-}{+}$, ${+}{-}{-}{-}$.

Let $v_i,v_j,v_k,v_l$, with $i<j<k<l$, be a $4$-tuple of vertices in $D$, inducing a drawing $H$ of $K_4$. By deforming the plane, we may assume that $v_i=(0,0)$, $v_l=(1,0)$, and that the vertices $v_j,v_k$ and the interiors of all six edges of $H$ lie in the interior of the strip between the vertical lines passing through $v_i$ and $v_l$. Note, however, that $H$ is not necessarily deformable to an $x$-bounded drawing with $v_j$ to the left of $v_k$: see Figure~\ref{fig_shellable_monotone_K4}, left.

Due to symmetry, we may assume that $\sigma(i,j,l)={+}$. That is, the vertex $v_j$ and the interiors of the edges $v_iv_j$ and $v_jv_l$ lie below the edge $v_iv_l$.
Now if $\sigma(i,k,l)={-}$, then the vertex $v_k$ and the interiors of the edges $v_iv_k$ and $v_kv_l$ lie above the edge $v_iv_l$. See Figure~\ref{fig_shellable_K4} a). Thus, the edges $v_iv_l$ and $v_jv_k$ are forced to cross an odd number of times, and no other pair of edges in $H$ cross. Also, the triangle $v_iv_jv_k$ is drawn counter-clockwise and the triangle $v_jv_kv_l$ clockwise, so we have $\sigma(i,j,k,l)={+}{+}{-}{-}$. Therefore, the $4$-tuple $v_i,v_j,v_k,v_l$ is odd in both drawings $D$ and $D'$.

\begin{figure}
 \begin{center}
   \includegraphics[scale=1]{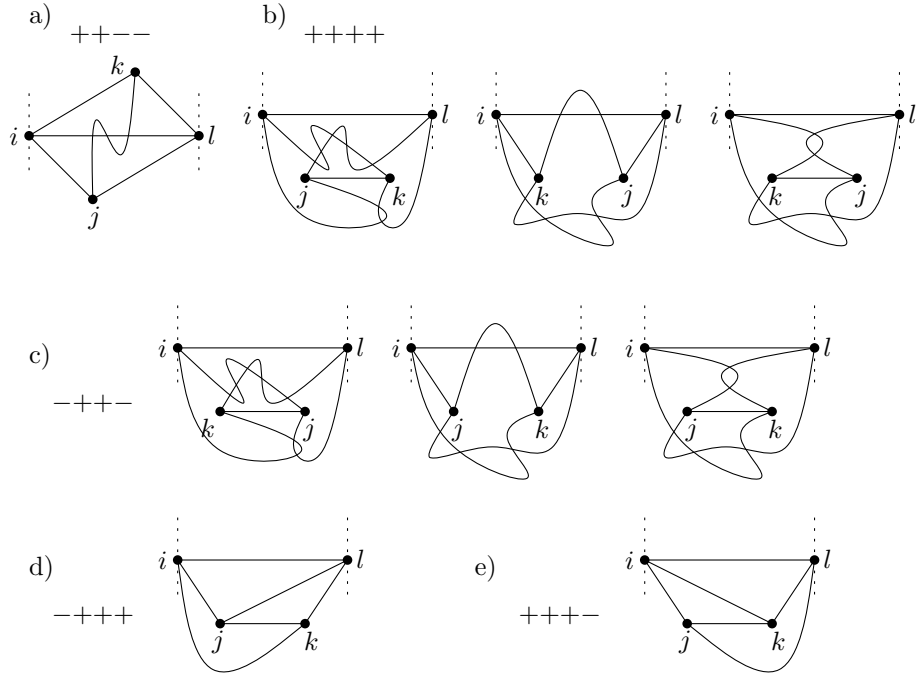}
  \caption{Examples of semisimple shellable drawings of $K_4$.}
 \label{fig_shellable_K4}
 \end{center}
\end{figure}

If $\sigma(i,k,l)={+}$, then the vertex $v_k$ and the interiors of the edges $v_iv_k$ and $v_kv_l$ lie below the edge $v_iv_l$. We have four cases according to the signs $\sigma(i,j,k)$ and $\sigma(j,k,l)$, which determine the vertical order of the edges near $v_i$ and $v_l$, respectively, but do not determine completely which edges cross oddly. This is true even when the drawing $H$ is simple; see Figure~\ref{fig_shellable_monotone_K4}. If $\sigma(i,j,k,l)={+}{+}{+}{+}$ or $\sigma(i,j,k,l)={-}{+}{+}{-}$, then either the edges $v_iv_k$ and $v_jv_l$ cross oddly, or the edges $v_iv_j$ and $v_kv_l$ cross oddly, and some other pair of edges may cross evenly; see Figure~\ref{fig_shellable_K4} b), c). In both cases, the $4$-tuple $v_i,v_j,v_k,v_l$ is odd in both drawings $D$ and $D'$. If $\sigma(i,j,k,l)={-}{+}{+}{+}$ or $\sigma(i,j,k,l)={+}{+}{+}{-}$, then no two edges cross; see Figure~\ref{fig_shellable_K4} d), e). In these last two cases, the $4$-tuple $v_i,v_j,v_k,v_l$ is even in both drawings $D$ and $D'$.
\end{proof}

\begin{figure}
 \begin{center}
   \includegraphics[scale=1]{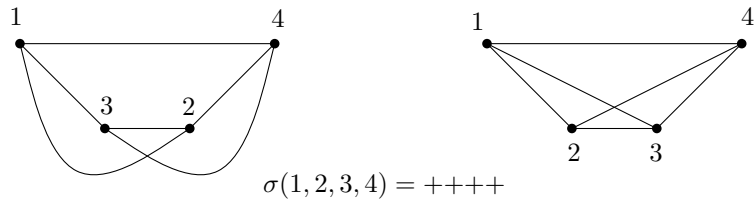}
  \caption{Two drawings of $K_4$ with the same order type.}
 \label{fig_shellable_monotone_K4}
 \end{center}
\end{figure}

Proposition~\ref{prop_shellable_semisimple} can be generalized to weakly semisimple shellable drawings, but the equality of the odd crossing numbers has to be replaced by inequality, since there are weakly semisimple shellable drawings of $K_4$ with odd crossing number $2$; see Figure~\ref{fig_weakly_shell_K4_ocr2}, left.

\begin{figure}
 \begin{center}
   \includegraphics[scale=1]{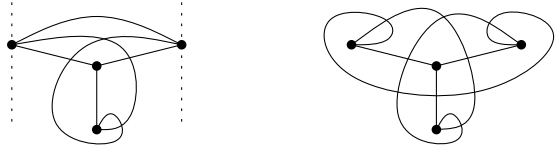}
  \caption{Left: a weakly semisimple shellable drawing of $K_4$ with two pairs of edges crossing oddly. Right: a weakly semisimple drawing of $K_4$ with three pairs of edges crossing oddly.}
 \label{fig_weakly_shell_K4_ocr2}
 \end{center}
\end{figure}

For general weakly semisimple drawings, the triangles are not necessarily simple closed curves. Nevertheless, we may still define the orientation of a triangle when every two of its edges cross evenly. Let $uvw$ be a triangle in a weakly semisimple drawing $D$ of $K_n$. Orient the closed curve $\gamma$ representing the triangle $uvw$ so that it passes through the vertices $u,v,w$ in this cyclic order. Then for each point $p$ on $\gamma$ that is not a crossing, a sufficiently small neighborhood of $p$ is divided by $\gamma$ into the {\em right neighborhood\/} and the {\em left neighborhood} of $p$, consistently with the chosen orientation of $\gamma$. 

Let $x$ be a point in the complement of $\gamma$ in the plane. The {\em winding number of $\gamma$ around $x$}
%, denoted by $w(\gamma,x)$, 
is, informally speaking, the number of counter-clockwise turns of $\gamma$ around $x$. More formally, if $\gamma$ is parametrized by continuous polar coordinates $(r(t), \varphi(t)): [0,1]\rightarrow (0,\infty)\times \mathbb{R}$, with center at $x$, then 
%$w(\gamma,x)=
the winding number of $\gamma$ around $x$ is $\frac{\varphi(1)-\varphi(0)}{2\pi}$. We use only the parity of the winding number, which is independent of the chosen orientation of $\gamma$.

We say that the triangle $uvw$, represented by the curve $\gamma$, is oriented {\em counter-clockwise\/} if for some point $x$ in the right neighborhood of $u$, the winding number of $\gamma$ around $x$ is even. Similarly, the triangle $uvw$ is oriented {\em clockwise\/} if the winding number of $\gamma$ around $x$ is odd. Due to the fact that every two edges of $uvw$ cross an even number of times, the definition does not change if we choose $x$ in the right neighborhood of $v$ or $w$. 
%we could as well choose $x$ in the right neighborhood of $v$ or $w$
%the definition does not depend on the particular choice of the vertex $u$ ( if we . 
We may thus generalize the notion of the {\em order type\/} to every weakly semisimple drawing of $K_n$ with vertices labeled $v_1, v_2, \dots, v_n$.

\begin{proposition}\label{prop_shellable_weakly_semisimple}
Let $D$ be a weakly semisimple shellable drawing of $K_n$. There is a semisimple $x$-monotone drawing $D'$ of $K_n$ with $\text{ocr}(D')\le\text{ocr}(D)$. 
\end{proposition}

\begin{proof}
We proceed in the same way as in the proof of Proposition~\ref{prop_shellable_semisimple}.
Let $v_1, v_2, \dots, v_n$ be a shelling sequence of $D$ and let $\sigma$ be the order type of $D$.
The fact that $\sigma$ satisfies the assumptions of Theorem~\ref{theorem_classif_semisimple} can be proved exactly in the same way as in the proof of Proposition~\ref{prop_shellable_semisimple}. Let $D'$ be a semisimple monotone drawing with signature function $\sigma$.

To prove the inequality, it is sufficient to show that every $4$-tuple of vertices in $D$ that induces a $K_4$ subgraph with odd crossing number $0$, corresponds to a $4$-tuple with no crossing in $D'$. For that, we only need to show that the $4$-tuple in $D$ is of the type ${+}{+}{+}{-}$, ${-}{+}{+}{+}$, ${-}{-}{-}{+}$ or ${+}{-}{-}{-}$. All other $4$-tuples in $D$ induce subgraphs with odd crossing number $1$ or $2$, which is at least as large as the odd crossing number of any $K_4$ subgraph in $D'$.

Let $v_i,v_j,v_k,v_l$, with $i<j<k<l$, be vertices in $D$ inducing a subgraph $H$ with all pairs of edges crossing evenly. We will show that there is a planar drawing $H''$ of the complete graph with vertices $v_i,v_j,v_k,v_l$, with $v_i$ and $v_l$ on its outer face, such that the orientation of each triangle in $H''$ is the same as in $H$. This will finish the proof, since such a drawing $H''$ is homeomorphic to one of the drawings in Figure~\ref{fig_shellable_K4} d), e).

The drawing $H$ satisfies the assumptions of the weak Hanani--Tutte theorem\mydotl\mycite{S13+_Tutte}\mydotr The weak Hanani--Tutte theorem says that for every drawing $D$ of a graph $G$ in the plane where every two edges cross an even number of times, there is a planar drawing $D'$ of $G$ which has the same rotation system as $D$ (that is, the cyclic orders of the edges around each vertex are preserved). The shortest proof of the weak Hanani--Tutte theorem, based on a more general version for arbitrary surfaces\mycommal\mycite{PSS09_surfaces}\mycommar was given by Fulek {\em et al.}\mycite[Lemma 3]{FPSS12_adjacent}\mydotr

We may assume that $v_i$ is the unique point in $H$ with smallest $x$-coordinate and that $v_l$ is the unique point in $H$ with largest $x$-coordinate. We extend the drawing $H$ to a drawing $K$ by adding a vertex $y$ placed below $H$, a vertex $z$ placed above $H$, and adding four edges $v_iy, yv_l, v_iz, zv_l$, drawn as monotone curves and forming a simple cycle $v_iyv_lz$. The cycle $v_iyv_lz$ forms the boundary of the outer face of $K$. By the weak Hanani--Tutte theorem, there is a planar drawing $K'$ having the same rotation system as $K$. In particular, the cycle $v_iyv_lz$ bounds a face $F$ in $K'$. Without loss of generality, we may assume that $F$ is the outer face of $K'$. Let $H'$ be the drawing obtained from $K'$ by removing the vertices $y,z$ and their adjacent edges. Clearly, the drawings $H'$ and $H$ have the same rotation system, $H'$ has no crossings, and $v_i$ and $v_l$ are on the boundary of the outer face of $H'$. The orientation of triangles $v_iv_jv_k$, $v_iv_jv_l$ and $v_iv_kv_l$ is determined by the rotation at $v_i$, and the orientation of the triangle $v_jv_kv_l$ is determined by the rotation at $v_l$. It follows that $H$ and $H'$ have the same order type, and the proof is finished.
\end{proof}

An attempt to generalize the approach in Proposition~\ref{prop_shellable_semisimple} to general non-shellable drawings fails, for the following reason. If $v_1, v_2, \dots, v_n$ is a chosen ordering of the vertices which is not a shelling sequence, we can have a $4$-tuple $v_i, v_j,v_k,v_l$, with $i<j<k<l$, inducing a planar drawing of $K_4$ such that $v_i$ or $v_l$ is the only vertex not incident with the outer face. These $4$-tuples are of type ${+}{-}{+}{+}$, ${+}{+}{-}{+}$, ${-}{+}{-}{-}$, or ${-}{-}{+}{-}$. In monotone drawings, such $4$-tuples are not semisimple and, moreover, have monotone odd crossing number $2$.
On the other hand, this is the only obstacle in generalizing Proposition~\ref{prop_shellable_semisimple} to all simple drawings. Indeed, it is easy to see that all simple drawings of $K_4$ with one crossing and arbitrary ordering of the vertices are of type ${+}{+}{+}{+}$, ${+}{+}{-}{-}$, ${-}{-}{+}{+}$, ${+}{-}{-}{+}$, ${-}{+}{+}{-}$, or ${-}{-}{-}{-}$, and thus correspond to a simple monotone drawing of $K_4$ with one crossing. In fact, this is still true also for semisimple drawings, by the claim in the proof of Lemma~\ref{lemma2}.

We may thus generalize Proposition~\ref{prop_shellable_semisimple} and consequently Theorem~\ref{theorem_1} to every drawing of $K_n$ such that there is an ordering $v_1, v_2, \dots, v_n$ of its vertices such that for every $4$-tuple $v_i, v_j,v_k,v_l$, with $i<j<k<l$, inducing a planar drawing $H$ of $K_4$, the vertices $v_i$ and $v_l$ are on the outer face of $H$. We call such a drawing {\em weakly shellable\/}. Trivially, every drawing of $K_n$ with ${n \choose 4}$ crossings is weakly shellable, with arbitrary ordering of its vertices. %The class of weakly shellable drawings is perhaps too artificial. 

%%
%We may thus generalize Proposition~\ref{prop_shellable_semisimple} and consequently Theorem~\ref{theorem_1} to the following class of graphs. We call a drawing of a complete graph {\em weakly shellable\/} if there is an ordering $v_1, v_2, \dots, v_n$ of its vertices such that for every $4$-tuple $v_i, v_j,v_k,v_l$, with $i<j<k<l$, inducing a planar drawing $H$ of $K_4$, the vertices $v_i$ and $v_l$ are on the outer face of $H$. Trivially, every drawing of $K_n$ with ${n \choose 4}$ crossings is weakly shellable, with arbitrary ordering of its vertices. %The class of weakly shellable drawings is perhaps too artificial. 
%%

\begin{corollary}
Let $D$ be a semisimple weakly shellable drawing of $K_n$. There is a semisimple $x$-monotone drawing $D'$ of $K_n$ with $\text{ocr}(D')=\text{ocr}(D)$.
\end{corollary}

\begin{corollary}
Let $D$ be a semisimple weakly shellable drawing of $K_n$. Then $\text{ocr}(D)\ge Z(n).$
\end{corollary}

We note that there are simple drawings of complete graphs that are not weakly shellable. For example, the drawing $F_6$ of $K_6$ in Figure~\ref{obr2_K_6_K_10}, left, has the property that every vertex is the central vertex of a planar drawing of $K_4$ induced by some $4$-tuple of vertices. Moreover, by taking two disjoint copies $F_6$ and adding all remaining $36$ edges, we obtain a simple drawing of $K_{12}$ which will not become weakly shellable even if we change its outer face by an arbitrary sequence of edge flips.

By removing the central vertex in $F_6$ we obtain a weakly shellable simple drawing of $K_5$ that is not shellable. This shows that weakly shellable drawings are more general than shellable drawings.

%-----------------------------------------------------------------------------
\subsection{Local characterization of shellable drawings}

The definition of a shellable drawing of a complete graph involves testing a quadratic number of subgraphs. It is easy to see that only linearly many of the subgraphs are sufficient. 

\begin{observation}
A sequence of vertices $v_1, v_2, \dots, v_n$ is a shelling sequence of a drawing of a complete graph if and only if for every $i\in [n]$, the vertex $v_i$ is on the outer face of the two subgraphs induced by the subsets of vertices $\{v_1, v_2, \dots, v_i\}$ and $\{v_i, v_{i+1}, \dots, v_n\}$.
\end{observation}

In a similar spirit as in Theorem~\ref{theorem_classif_semisimple}, we may obtain a local characterization of shellable drawings, by testing only the subgraphs with four vertices. Like in Theorem~\ref{theorem_classif_semisimple}, we need to assume a fixed ordering of the vertices, as there are arbitrarily large minimal non-shellable (and non-monotone) drawings of complete graphs---for example, ``flowers'' generalizing the drawing $F_6$ in Figure~\ref{obr2_K_6_K_10}, left. Unlike in the case of monotone drawings, the order type does not necessarily determine a unique shellable drawing; see Figure~\ref{fig_shellable_monotone_K4}.

\begin{theorem}\label{theorem_local_char_shellable}
Let $D$ be a simple drawing of $K_n$. A sequence $v_1, v_2, \dots, v_n$ of the vertices is a shelling sequence of $D$ if and only if every $4$-tuple $v_i,v_j,v_k,v_l$, with $i<j<k<l$, induces a drawing of $K_4$ having $v_i$ and $v_l$ on its outer face.
\end{theorem}

To show Theorem~\ref{theorem_local_char_shellable}, we use the following generalization of Carath\'eodory's theorem.

\begin{lemma}[Carath\'eodory's theorem for simple complete topological graphs]\label{lemma_caratheodory}
Let $D$ be a simple drawing of $K_n$ and let $x$ be a point in the interior of a bounded face of $D$. Then there is a triangle $uvw$ in $D$ containing $x$ in its interior. Moreover, there is a set of at most $n-2$ triangles covering all bounded faces of $D$
and such that every edge of $D$ is in at most two of these triangles.
\end{lemma}

We use only the first part of the lemma. 
The stronger conclusions are included since they follow easily from the proof and might be interesting on their own.

\begin{proof}
We proceed by induction on the number of vertices. For $n\le 2$ the assumptions are vacuous and for $n=3$ the statement is obvious. Now let $n\ge 4$ and suppose that the lemma has been proved for drawings with at most $n-1$ vertices. Let $v_1, v_2, \dots, v_n$ be the vertices of $D$. Let $D_{n-1}$ be the drawing of the complete subgraph induced by $v_1, v_2, \dots, v_{n-1}$. Let $C$ be the simple curve forming the boundary of the outer face of $D_{n-1}$. By induction, all bounded faces of $D_{n-1}$ are covered by a set $\mathcal{T}_{n-1}$ of at most $n-3$ triangles 
%%%
so that no edge is contained in more than two triangles from $\mathcal{T}_{n-1}$. We assume (and prove) an even stronger induction statement: if two triangles from $\mathcal{T}_{n-1}$ share an edge $e$, then they do not cover the same face incident with $e$. That is, the two triangles are ``attached'' to $e$ from the opposite sides of $e$. 
%%%

By adding $v_n$ with its incident edges to $D_{n-1}$, the outer face of $D_{n-1}$ is partitioned into the outer face of $D_n$ and several bounded faces. We show that all these new bounded faces can be covered by a single triangle.
We distinguish two cases.

\smallskip
\noindent
%\paragraph{
{\bf a)} The vertex $v_n$ is in the outer face of $D_{n-1}$. First we observe that no edge $v_iv_n$ has more than one crossing with $C$. See Figure~\ref{fig_caratheodory} a). Suppose the contrary and let $x_1$ and $x_2$ be two crossings of $v_nv_i$ with $C$ closest to $v_n$. Then the portion of $v_nv_i$ between $x_1$ and $x_2$ separates the drawing $D_{n-1}$ into two parts, each of them containing at least one vertex. In particular, the part that does not contain $v_i$ contains some other vertex $v_j$. The edge $v_iv_j$ has to lie in the closed region bounded by $C$, thus it is forced to cross the edge $v_iv_n$; a contradiction.

It follows that for every edge $v_nv_i$, either the relative interior of $v_nv_i$ lies outside $C$ and $v_i$ lies on $C$, or $v_nv_i$ crosses $C$ in exactly one point, $x_i$, and the portion of $v_nv_i$ between $x_i$ and $v_i$ lies in the closed region bounded by $C$. In all cases, only the initial portion of the edge $v_nv_i$ lies in the outer face of $D_{n-1}$. Consequently, only two edges incident with $v_n$ are incident with the outer face of $D_n$.

Let $v_nv_k$ and $v_nv_l$ be the two edges incident with $v_n$ and with the outer face of $D_n$. Since the relative interior of the edge $v_kv_l$ lies inside $C$, the triangle $v_nv_kv_l$ covers all bounded faces of $D_n$ lying outside $C$. % and we let $\mathcal{T}_n=\mathcal{T}_{n-1}\cup\{v_nv_kv_l\}$.
%%%%
If no triangle from $\mathcal{T}_{n-1}$ has the edge $v_kv_l$, or if exactly one such triangle, $v_mv_kv_l$, exists but has the opposite orientation from $v_nv_kv_l$, we let $\mathcal{T}_n=\mathcal{T}_{n-1}\cup\{v_nv_kv_l\}$. If some triangle $v_mv_kv_l$ from $\mathcal{T}_{n-1}$ has the edge $v_kv_l$ and has the same orientation as $v_nv_kv_l$, then $v_m$ cannot lie outside $v_nv_kv_l$, as then the edge $v_nv_m$ would be incident with the outer face. Hence $v_m$ is inside $v_nv_kv_l$. The orientation of the triangle $v_mv_kv_l$ then implies that the whole triangle $v_mv_kv_l$ is covered by $v_nv_kv_l$, and so we let $\mathcal{T}_n=(\mathcal{T}_{n-1}\setminus\{v_mv_kv_l\})\cup\{v_nv_kv_l\}$.
%%%%
%\paragraph{

\begin{figure}
 \begin{center}
   \includegraphics[scale=1]{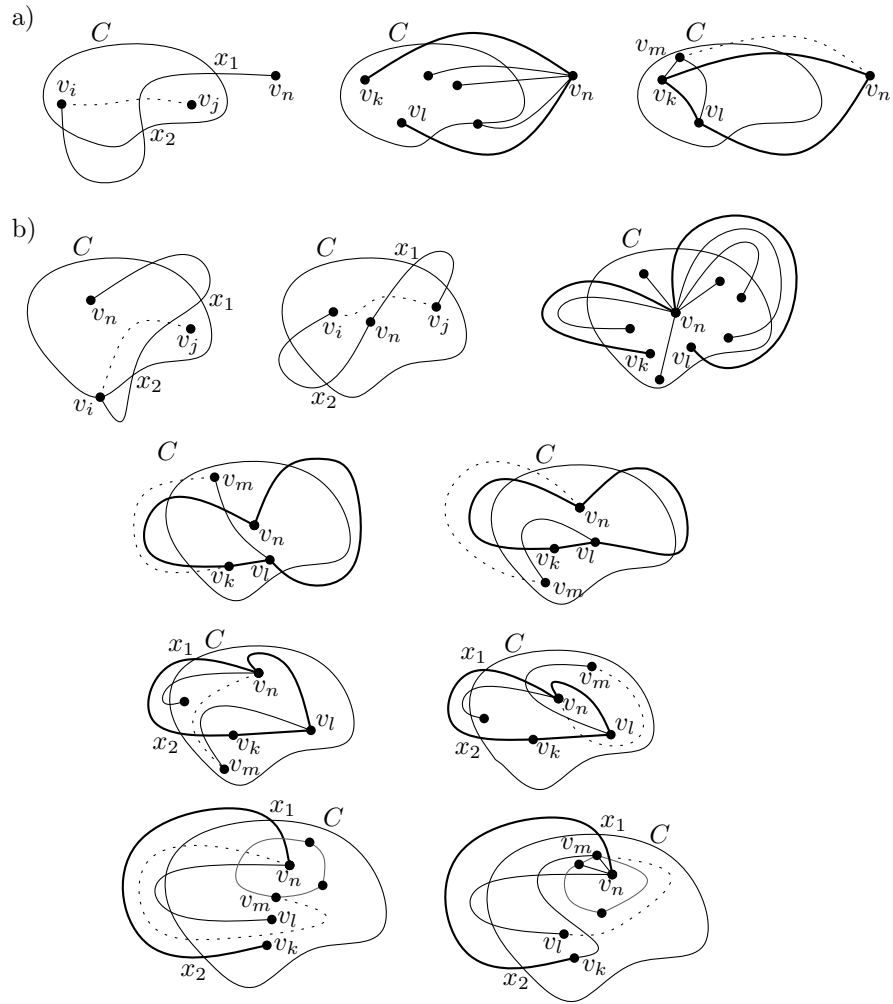}
  \caption{a) adding a vertex $v_n$ to the outer face. b) adding a vertex $v_n$ to a bounded face. Dotted curves represent forbidden edges.}
 \label{fig_caratheodory}
 \end{center}
\end{figure}

\smallskip
\noindent
{\bf b)} The vertex $v_n$ is in the interior of some bounded face of $D_{n-1}$. %}
By a similar argument as in part a), every edge $v_nv_i$ has at most two crossings with $C$. See Figure~\ref{fig_caratheodory} b). If no edge incident with $v_n$ is incident with the outer face of $D_n$, then $C$ is the boundary of the outer face of $D_n$ and thus we let $\mathcal{T}_n=\mathcal{T}_{n-1}$. If two edges $v_nv_i$ and $v_nv_j$ cross $C$, they separate the closed region bounded by $C$ into two parts. The vertices $v_i$ and $v_j$ must be in the same part, otherwise the edge $v_iv_j$ would cross $v_nv_i$ or $v_nv_j$, which is forbidden. 

It follows that at most two edges incident with $v_n$, $v_nv_k$ and, possibly, $v_nv_l$, are incident with the outer face of $D_n$. All other edges $v_nv_i$ that cross $C$ do so in a ``nested fashion'' in the interval bounded by the crossings of $v_nv_k$ with $C$, or in the interval bounded by the crossings of $v_nv_l$ with $C$; see Figure~\ref{fig_caratheodory} b). Hence, if $v_nv_k$ and $v_nv_l$ are incident with the outer face, then the triangle $v_nv_kv_l$ covers all bounded faces of $D_n$ that lie outside $C$. 

If there is no triangle $v_mv_kv_l$ in $\mathcal{T}_{n-1}$ with the same orientation as $v_nv_kv_l$, we let $\mathcal{T}_n=\mathcal{T}_{n-1}\cup\{v_nv_kv_l\}$. If there is a triangle $v_mv_kv_l$ in $\mathcal{T}_{n-1}$ with the same orientation as $v_nv_kv_l$, then $v_m$ has to be inside $v_nv_kv_l$. For if $v_m$ was outside $v_nv_kv_l$ in the region bounded by $v_nv_k$, $v_nv_l$ and $C$, then one of the edges $v_mv_k$ or $v_mv_l$ would be forced to cross an adjacent edge or $C$. Similarly, if $v_m$ was in the other region outside $v_nv_kv_l$ and inside (or on) $C$, then the edge $v_mv_n$ would be forced to cross an adjacent edge or it would separate $v_nv_k$ or $v_nv_l$ from the outer face. Like in case a), if $v_m$ is inside $v_nv_kv_l$, then the orientation of $v_mv_kv_l$ implies that $v_mv_kv_l$ is covered by $v_nv_kv_l$. We let $\mathcal{T}_n=(\mathcal{T}_{n-1}\setminus\{v_mv_kv_l\})\cup\{v_nv_kv_l\}$.

We are left with the case when $v_nv_k$ is the only edge incident with $v_n$ and with the outer face. 
Let $x_1$ and $x_2$ be the crossings of $v_nv_k$ with $C$, so that $x_1$ is between $v_n$ and $x_2$. Without loss of generality, assume that the portion of the edge $v_nv_k$ starting at $x_1$ and ending at $x_2$ is oriented counter-clockwise on the boundary of the outer face. Let $v_nv_l$ be the edge following $v_nv_k$ clockwise in the rotation at $v_n$. 

If $v_nv_l$ does not cross $C$, then the triangle $v_nv_kv_l$ covers all bounded faces of $D_n$ outside $C$.
Similarly as in the previous case, we argue that if there is a triangle $v_mv_kv_l$ in $\mathcal{T}_{n-1}$ with the same orientation as $v_nv_kv_l$, then $v_m$ is inside $v_nv_kv_l$ and so $v_mv_kv_l$ is covered by $v_nv_kv_l$, otherwise the edge $v_nv_m$ would have have to cross some adjacent edge. Here we use the fact that no edge leaves $v_n$ outside the triangle $v_nv_kv_l$. Again, we let $\mathcal{T}_n=\mathcal{T}_{n-1}\setminus\{v_mv_kv_l\}\cup\{v_nv_kv_l\}$ or $\mathcal{T}_n=\mathcal{T}_{n-1}\cup\{v_nv_kv_l\}$, according to the existence of the triangle $v_mv_kv_l$ covered by $v_nv_kv_l$.

Finally, suppose that $v_nv_l$ crosses $C$. By induction, there is a triangle $v_mv_iv_j \in \mathcal{T}_{n-1}$ containing $v_n$ in its interior. Hence, each of the edges $v_nv_k$ and $v_nv_l$ crosses at least one edge of $v_mv_iv_j$. If $v_nv_k$ and $v_nv_l$ cross the same edge, say, $v_mv_i$, then the edge $v_nv_m$ also crosses $v_mv_i$, a contradiction. Otherwise, the region bounded by $v_nv_k$, $v_nv_l$ and $C$ that does not contain $x_2$, contains at least one vertex of the triangle $v_mv_iv_j$, say, $v_m$. Then it is impossible to draw the edges $v_mv_k$ and $v_mv_l$ so that the resulting drawing is simple. Therefore $v_nv_l$ cannot cross $C$ and we are finished.
\end{proof}

\begin{proof}[Proof of Theorem~\ref{theorem_local_char_shellable}]
The condition on $4$-tuples is clearly necessary. We show that it is also sufficient. Suppose that $D$ is a simple drawing of $K_n$ and that for some $i\in[n]$, the vertex $v_i$ is not incident with the outer face of the subgraph induced by the subset $\{v_1, v_2, \dots, v_i\}$ (the case with the subset $\{v_{i}, v_{i+1}, \dots, v_n\}$ is symmetric). By Lemma~\ref{lemma_caratheodory}, there is a triangle $v_jv_kv_l$ with $1\le j<k<l<i$ containing $v_i$ in its interior. In particular, $v_i$ is not incident with the outer face of the drawing of $K_4$ induced by the $4$-tuple $v_j,v_k,v_l,v_i$.
\end{proof}

%-----------------------------------------------------------------------------
\subsection{Shellable drawings and monotone drawings}\label{sub_shellable_not_monotone}
Here we show that shellable drawings form a more general class than monotone drawings. We also show how monotone drawings may be characterized as a special case of shellable drawings.

Two drawings $D_1$, $D_2$ of a graph $G=(V,E)$
%with vertices labeled $v_1,v_2,\dots,v_n$ 
are {\em weakly isomorphic\/} if for every two edges $e, f \in E$, $e$ and $f$ cross in $D_1$ if and only if they cross in $D_2$. 
Let $D$ be a simple drawing of $K_n$ with vertex set $\{v_1,v_2,\dots ,v_n\}$. We say that a sequence of vertices $v_1,v_2,\dots,v_n$ is an {\em $x$-monotone sequence\/} of $D$ if $v_1$ and $v_n$ are incident with the outer face of $D$ and $D$ is weakly isomorphic to a simple monotone drawing where $v_i=(i,0)$ for every $i\in [n]$.

We have the following characterization of $x$-monotone sequences in terms of shelling sequences.

\begin{lemma}
Let $D$ be a simple drawing of $K_n$. A sequence of vertices $v_1,v_2,\dots,v_n$ is an $x$-monotone sequence of $D$ if and only if it is a shelling sequence of $D$ and the path $v_1v_2\dots v_n$ does not cross itself.
\end{lemma}

\begin{proof}
The ``only if'' part is obvious. Let $v_1,v_2,\dots , v_n$ be a shelling sequence such that the path $v_1v_2\dots v_n$ does not cross itself. We claim that for every $v_i,v_j,v_k,v_l$ with $1\le i<j<k<l\le n$, the path $v_iv_jv_kv_l$ does not cross itself. 
%This is obvious if the drawing $H$ of $K_4$ induced by the vertices $v_i,v_j,v_k,v_l$ is planar.
Let $H$ be the drawing of $K_4$ induced by the vertices $v_i,v_j,v_k,v_l$.

\begin{figure}
 \begin{center}
   \includegraphics[scale=1]{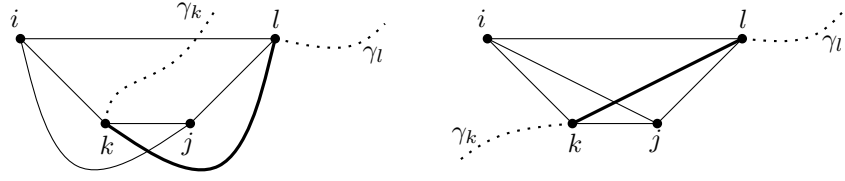}
  \caption{Drawings of $K_4$ with shelling sequence $v_i,v_j,v_k,v_l$ where the edges $v_iv_j$ and $v_kv_l$ cross. Dotted curves cannot cross the path $P_{i,j}$. Thick curves cannot cross the path $P_{i,j}$ either except for the case $j=i+1.$}
 \label{fig_shell_noncrossing_k4}
 \end{center}
\end{figure}

Suppose for contrary that the path $v_iv_jv_kv_l$ in $H$ crosses itself. That is, the edges $v_iv_j$ and $v_kv_l$ cross. Let $i,j,k,l$ be such a $4$-tuple with the pair $(l-i,j-i)$ lexicographically smallest.
The drawing $H$ is homeomorphic to one of the drawings in Figure~\ref{fig_shell_noncrossing_k4}. 
Since $v_l$ is on the outer face of the complete subgraph $D_{i,l}$ with vertices $v_i,v_{i+1},\dots,v_l$, there is an unbounded curve $\gamma_l$ starting at $v_l$ going to infinity and avoiding all edges of $D_{i,l}$. Similarly, there is an unbounded curve $\gamma_k$ starting at $v_k$ going to infinity and avoiding all edges of the complete graph induced by $v_i,v_{i+1},\dots,v_k$. In particular, $\gamma_l$ and $\gamma_k$ do not cross the path $P_{i,j}=v_iv_{i+1}\dots v_j$, the curve $\gamma_l$ lies completely in the outer face of $H$, and $\gamma_k$ lies completely outside the triangle $v_iv_jv_k$. 

% nor the edges $v_iv_k$, $v_kv_j$.

By the minimality of $l-i$, the edge $v_kv_l$ crosses no edge $v_{i+a}v_{i+a+1}$ with $1\le a\le j-i-1$. By the minimality of $j-i$, the edge $v_iv_{i+1}$ does not cross $v_kv_l$, unless $j=i+1$. 
Since the double-infinite curve formed by $\gamma_k,v_kv_l$ and $\gamma_l$ separates $v_i$ from $v_j$, it must cross the path $P_{i,j}$. This implies that $j=i+1$.

Similarly, there are unbounded curves $\gamma_i$ and $\gamma_j$ starting at $v_i$ and $v_j$, respectively, that do not cross the path $P_{k,l}=v_kv_{k+1}\dots v_l$, the curve $\gamma_i$ lies completely in the outer face of $H$, and $\gamma_j$ lies completely outside the triangle $v_jv_kv_l$. Since $j=i+1$ and by the assumption, the edge $v_iv_j$ does not cross the path $P_{k,l}$ either. The double-infinite curve formed by $\gamma_i, v_iv_j$ and $\gamma_j$ thus separates $v_k$ from $v_l$ but does not cross $P_{k,l}$; this is a contradiction.

Since the path $v_iv_jv_kv_l$ does not cross itself, the order type of $H$ determines the drawing up to an isotopy. Indeed, the drawings in Figure~\ref{fig_shellable_monotone_K4} represent, up to relabeling, the only two isotopy classes of simple shellable drawings of $K_4$ that have the same order type. There are two possible shelling sequences common for both drawings. For the shelling sequence $1,2,3,4$, the corresponding path is noncrossing only in the right drawing. For the shelling sequence $1,3,2,4$, the corresponding path is noncrossing only in the left drawing.

Let $\sigma$ be the order type of $D$. By the proof of Proposition~\ref{prop_shellable_semisimple}, there is a semisimple monotone drawing $D'$ with signature function $\sigma$ such that two edges cross oddly in $D$ if and only if they cross oddly in $D'$.
%%for every $4$-tuple $v_i,v_j,v_k,v_l$, the drawings of the complete graph induced by $v_i,v_j,v_k,v_l$ in $D$ and in $D'$ the same set of pairs of edges crossing oddly. 
%%In particular, the drawings $D$ and $D'$ are weakly isomorphic.

It remains to show that $D'$ is simple. By Theorem~\ref{theorem_classif_semisimple}, it is sufficient to show that there is no $5$-tuple $(a,b,c,d,e)$ with $a<b<c<d<e$ such that $\sigma(a,b,e)=\sigma(a,d,e)=\sigma(b,c,d)=\overline{\sigma(a,c,e)}=\xi$, where $\xi \in \{+,-\}$. Suppose for contrary that there is such a $5$-tuple. By symmetry, we may assume that $\xi=+$. The vertices $v_a,v_b,v_c,v_d,v_e$ induce a shellable drawing $K$ of $K_5$ in $D$. We may deform the plane by an isotopy so that $v_a=(0,0), v_e=(1,0)$, and so that all edges of $K$ are drawn between the vertical lines going through $v_a$ and $v_e$. From $\sigma(a,b,e)= +$ and $\sigma(a,c,e) = -$ we have $\sigma(a,b,c,e)={+}{+}{-}{-}$. Similarly, from $\sigma(a,c,e)=-$ and $\sigma(a,d,e)=+$ we have $\sigma(a,c,d,e)={-}{-}{+}{+}$. This further implies that $\sigma(a,b,c,d)={+}{-}{-}{+}$. In particular, the edges $v_av_c$ and $v_bv_d$ cross. The signatures also imply that $v_b$ and $v_d$ are below the edge $v_av_e$ and $v_c$ is above the edge $v_av_e$. For a simple drawing this means that the edge $v_bv_d$ is below $v_av_e$ and the relative interior of the edge $v_av_c$ is above $v_av_e$, therefore the edges $v_av_c$ and $v_bv_d$ cannot cross; a contradiction.
\end{proof}

One may notice the following apparent difference between $x$-monotone and shellable sequences: some drawings of $K_n$ have much more shellable sequences than $x$-monotone sequences. For example, for the convex geometric drawing of $K_n$, all $n!$ permutations of vertices are shelling sequences, whereas at most $n\cdot 2^{n-2}$ permutations of vertices, inducing a noncrossing Hamiltonian path, are $x$-monotone sequences. 

To show that shellable drawings are indeed more general than monotone drawings, we provide an example of a shellable drawing that has no $x$-monotone sequence.

\begin{theorem}
The drawing in Figure~\ref{fig_shellable_vs_monotone} is a simple shellable drawing of $K_9$ which is not weakly isomorphic to a simple monotone drawing.
\end{theorem}

\begin{figure}
 \begin{center}
   \includegraphics[scale=1]{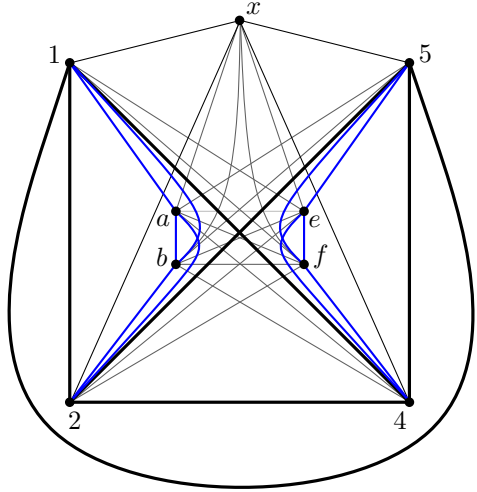}
  \caption{A simple drawing $S_9$ of $K_9$ with shelling sequence $1,4,f,e,x,a,b,2,5$ which has no $x$-monotone sequence.}
 \label{fig_shellable_vs_monotone}
 \end{center}
\end{figure}

\begin{proof}
Clearly, the sequence $1,4,f,e,x,a,b,2,5$ is a shelling sequence of the drawing $S_9$ in Figure~\ref{fig_shellable_vs_monotone}. Suppose that $\mu$ is an $x$-monotone sequence of $S_9$. We write $v \prec w$ for vertices $v,w$ if $v$ precedes $w$ in $\mu$. By symmetry, we may assume that $1\prec 5$. The subgraphs induced by $4$-tuples $\{1,2,4,5\}$, $\{1,2,a,b\}$ and $\{4,5,e,f\}$ have unique $x$-monotone sequences, up to reversal. In particular, we have $1\prec 2 \prec 4 \prec 5$, which in turn implies that $1\prec a \prec b \prec 2 \prec 4 \prec f \prec e \prec 5$. To uncover the vertex $a$, it is not sufficient to remove the vertex $1$, we have to remove at least one more vertex. Since all vertices except for $x$ are preceded by $a$ in $\mu$, we have $x\prec a$. Similarly, to uncover the vertex $e$, it is not sufficient to remove the vertex $5$, and the only available vertex is $x$. Therefore, $e \prec x$. These conditions cannot be fulfilled, thus $S_9$ has no $x$-monotone sequence.
\end{proof}

%-------------------------------------------------------------------------------------
\subsection{Crossing number and $k$-edges in weakly semisimple drawings}
Here we show a generalization of Lemma~\ref{lemma2} to weakly semisimple drawings, which may be used to generalize Theorem~\ref{theorem_1} and the result of \'Abrego {\em et al.}\mycite{aamrs13_shellable} to weakly semisimple $s$-shellable drawings with $s\ge n/2$. As in Proposition~\ref{prop_shellable_weakly_semisimple}, the equality has to be replaced by an inequality.
Since the orientation of triangles and hence the order type can be still defined in weakly semisimple drawings (see the definition before Proposition~\ref{prop_shellable_weakly_semisimple}), the notions of $k$-edges, ${\le}k$-edges, ${\le}{\le}k$-edges and separations  generalize to weakly semisimple drawings as well.

\begin{lemma}
\label{lemma3}
For every weakly semisimple drawing $D$ of $K_n$ we have
\[
{\rm ocr}(D) \ge \;2 \sum_{k=0}^{\lfloor n/2 \rfloor-2} {E_{{\le}{\le}k}(D)}-\frac{1}{2}{n \choose 2}\left\lfloor\frac{n-2}{2}\right\rfloor -\frac{1}{2}\left(1+(-1)^n\right)E_{{\le}{\le}\lfloor n/2 \rfloor-2}(D).
\]
\end{lemma}

\begin{proof}
The lemma follows in the same way as Lemma~\ref{lemma2} or Lemma~\ref{lemma1}, after proving that every weakly semisimple drawing $D$ of $K_4$ satisfies the inequality ${\rm ocr}(D)+E_1(D)\ge 3$.
The equality is not always attained as there are weakly semisimple drawings of $K_4$ with odd crossing number $3$ and with six separations; see Figure~\ref{fig_weakly_shell_K4_ocr2}, right.

Let $D$ be a weakly semisimple drawing of $K_4$. The {\em separation graph\/} of $D$ is the subgraph of $D$ formed by the $1$-edges in $D$.
The separation graph depends only on the order type of $D$. Every order type can be obtained from each other by changing the orientation of some triangles. By changing the orientation of a triangle $uvw$, the edges $uv, uw, vw$ change from $0$-edges to $1$-edges and vice versa. It follows that the degree of each vertex in the separation graph either remains the same or changes by $2$. Since in the planar drawing of $K_4$ the separation graph is isomorphic to $K_{1,3}$, it follows that the separation graph of $D$ has all vertices of odd degree. That is, it is isomorphic to $K_2+K_2, K_{1,3}$, or $K_4$. In particular, $E_1(D)\ge 2$. 

Therefore, the inequality is proved for drawings with ${\rm ocr}(D)\ge 1$. Now suppose that ${\rm ocr}(D)=0$. We show that the separation graph of $D$ is isomorphic to $K_{1,3}$. We achieve this by transforming $D$ into a drawing $D''$ by a sequence of edge flips and then to a planar drawing $D'$ which has the same order type as $D''$. Performing the steps in reverse order will imply that the separation graph of each of $D', D''$ and $D$ is isomorphic to $K_{1,3}$.

Every edge flip (see the definition in the proof of Lemma~\ref{lemma2}) in a drawing of $K_4$ changes the orientation of two adjacent triangles. The separation graph is thus transformed by taking the symmetric difference with a cycle $C_4$. Clearly, if the separation graph is isomorphic to $K_{1,3}$, then its symmetric difference with arbitrarily positioned $C_4$ is isomorphic to $K_{1,3}$ as well. We may transform $D$ by a sequence of edge flips into a drawing $D''$ which has at least one vertex $v$ on the outer face. Let $w_1,w_2,w_3$ be the other three vertices of $D''$, so that the initial portions of the edges $vw_1$ and $vw_3$ are incident with the outer face of $D''$ and the rotation at $v$ is $w_1, w_2, w_3$. See Figure~\ref{fig_ht_blackbox}, left.

\begin{figure}
 \begin{center}
   \includegraphics[scale=1]{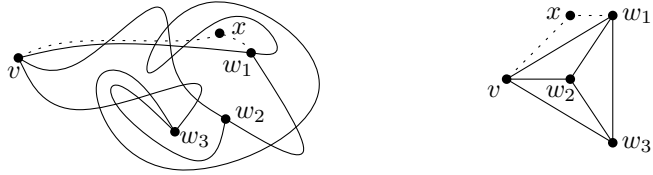}
  \caption{Left: adding an auxiliary vertex and two edges to a drawing of $K_4$ before applying the weak Hanani--Tutte theorem. Right: a planar drawing of the extended graph with the same rotation system.}
 \label{fig_ht_blackbox}
 \end{center}
\end{figure}

We extend the drawing $D''$ by adding one auxiliary vertex $x$ close to $w_1$ and edges $vx$ and $xw_1$, so that $x$ follows immediately after $v$ in the rotation at $w_1$, the rotation at $v$ is $x, w_1, w_2, w_3$, the triangle $vxw_1$ is oriented clockwise and the path $vxw_1$ is drawn close to the edge $vw_1$. We denote this new drawing as $K$.

Since every two edges cross evenly in $D''$, the same is true for the drawing $K$ and thus we may apply the weak Hanani--Tutte theorem to $K$. We obtain a planar drawing $K'$ with the same rotation system as $K$. We may assume that $vxw_1w_3$ forms a boundary of the outer face of $K'$. See Figure~\ref{fig_ht_blackbox}, right. Let $D'$ be the subgraph of $K'$ obtained after removing $x$ and its adjacent edges. The orientations of all three triangles incident to $v$ are the same in $D''$ and in $D'$, since $v$ is on the outer face in both drawings and the rotation at $v$ is the same in $K$ and in $K'$.

It remains to compare the orientation of the triangle $w_1w_2w_3$ in $D''$ and $D'$. 
Let $\gamma$ ($\gamma'$) be the closed curve formed by the edges of the triangle $w_1w_2w_3$ in $K$ ($K'$, respectively). Since the curve $vx$ crosses every edge of $D''$ an even number of times, the winding number of $\gamma$ around $x$ has the same parity as the winding number of $\gamma$ around $v$. Since $v$ is in the outer face of $D''$, both winding numbers are even. Since $x$ is outside $\gamma'$ in $K'$, the winding number of $\gamma'$ around $x$ is even as well. Together with the fact that in both drawings $K$ and $K'$, the rotation at $w_1$ is the same, this implies that the triangle $w_1w_2w_3$ is oriented counter-clockwise in both drawings. Therefore, $D''$ and $D'$ have the same order type.
\end{proof}

Combining Lemma~\ref{lemma3} with the proof by \'Abrego {\em et al.}\mycommal\mycite{aamrs13_shellable}\mycommar we obtain the following generalization.

\begin{corollary}
Let $s\ge n/2$ and let $D$ be a weakly semisimple $s$-shellable drawing of $K_n$. Then $\text{ocr}(D)\ge Z(n).$
\end{corollary}

%==========================================================================================================

\section{Concluding remarks}

It would be interesting to see if techniques similar to those used in the proof of Theorem~\ref{theorem_1} can be used to prove Hill's conjecture for general drawings of complete graphs. We note that the same approach does not generalize to all drawings. For example, a particular planar realization of the so-called {\em cylindrical drawing\/}\mycite{guy60,HH63_on_number} of $K_{10}$, with crossing number $Z(10)$, does not satisfy the lower bound on ${\le\le}1$-edges from Theorem~\ref{thm2}. See Figure~\ref{obr2_K_6_K_10}, right. Figure~\ref{obr2_K_6_K_10}, left, shows an even smaller example, but this drawing of $K_6$ is not crossing optimal. Analogous cylindrical drawings of $K_{4k+6}$, for $k\ge 2$, violate the lower bound on ${\le}{\le}k$-edges from Theorem~\ref{thm2}.

\begin{figure}
 \begin{center}
  \includegraphics{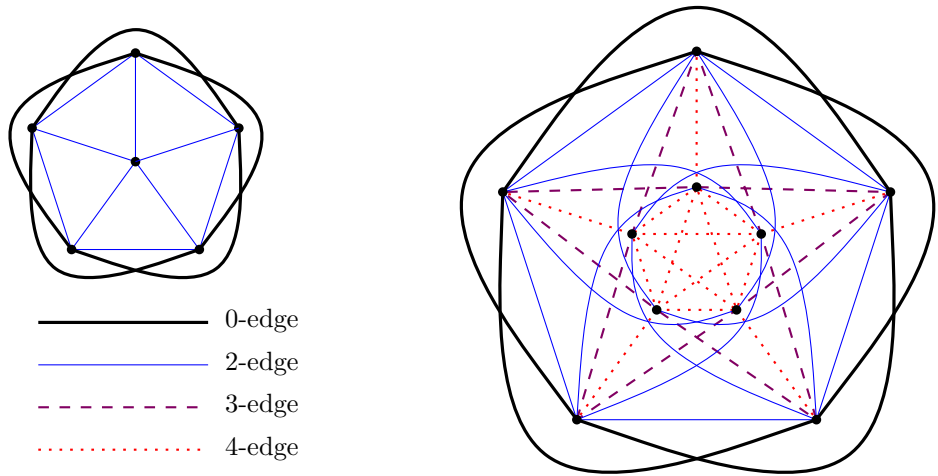}
  \caption{A general simple drawing of $K_{6}$ (left) and a cylindrical drawing of $K_{10}$ (right) where $E_{0}=5$ and $E_{1}=0$, hence $E_{{\le}{\le}1}=10<12=3{{1+3}\choose{3}}$.}
 \label{obr2_K_6_K_10}
 \end{center}
\end{figure}

Extrapolating the definitions of ${\le}k$-edges and ${\le}{\le}k$-edges, we
%call every ${\le}{\le}j$-edge for some $j \le k$ a {\em ${\le}{\le}k$-edge}
define the number of {\em ${\le}{\le}{\le}k$-edges}, $E_{{\le}{\le}{\le}k}(D)$, by the following identity.
\[
E_{{\le}{\le}{\le}k}(D) = \sum_{j=0}^{k}E_{{\le\le}j}(D) =\sum_{i=0}^k {k+2-i\choose 2}E_{i}(D).
\]
In our context, using ${\le}{\le}{\le}k$-edges seems to be even more natural than using ${\le}{\le}k$-edges, since the formula from Lemma~\ref{lemma1} can be rewritten in the following compact form:
\begin{align*}
{\rm cr}(D) &= 2 E_{{\le}{\le}{\le}\lfloor n/2 \rfloor-2}(D) -\frac{1}{8}n(n-1)(n-3) \hskip 1mm \text{ for } n \text{ odd, and} \\
{\rm cr}(D) &= E_{{\le}{\le}{\le}\lfloor n/2 \rfloor-3}(D) + E_{{\le}{\le}{\le}\lfloor n/2 \rfloor-2}(D) -\frac{1}{8}n(n-1)(n-2) \hskip 1mm \text{ for } n \text{ even}.
\end{align*}
We conjecture that the following lower bound on ${\le}{\le}{\le}k$-edges is satisfied by all simple drawings of complete graphs.

\begin{conjecture}\label{conjecture_nase}
Let $n \ge 3$ and let $D$ be a simple drawing of $K_n$. Then for every $k$ satisfying $0 \le k < n/2 - 1$, we have
\[E_{{\le}{\le}{\le}k}(D) \ge 3{k+4 \choose 4}.\]
\end{conjecture}

Conjecture~\ref{conjecture_nase} is stronger than Hill's conjecture.
Theorem~\ref{thm2} implies Conjecture~\ref{conjecture_nase} for all simple $x$-monotone drawings.
All our examples of simple drawings of complete graphs, including the cylindrical drawings, also satisfy Conjecture~\ref{conjecture_nase}. We note that Conjecture~\ref{conjecture_nase} is trivially satisfied for $k=0$, since every simple drawing of a complete graph with at least three vertices has at least three $0$-edges---those incident with the outer face.

We have no counterexample even to the following conjecture, which further generalizes Conjecture~\ref{conjecture_nase} to arbitrary graphs. 

\begin{conjecture}\label{conjecture_nase_2}
Let $k \ge 0$ and let $D$ be a simple drawing of a graph with at least ${2k+3 \choose 2}$ edges. Then 
\[E_{{\le}{\le}{\le}k}(D) \ge 3{k+4 \choose 4}.\]
\end{conjecture}

Note that in a drawing of a general graph with $n$ vertices, a $k$-edge contained in $t$ triangles is also a $(t-k)$-edge, but not necessarily an $(n-2-k)$-edge. Thus, for example, in every drawing of a triangle-free graph, every edge is a $0$-edge. This suggests that it might be easier to prove 
Conjecture~\ref{conjecture_nase_2} for non-complete graphs. Also, Conjecture~\ref{conjecture_nase_2} or some still stronger variant might be susceptible to a proof by induction on the number of edges.

%It would also be interesting to further generalize Theorem~\ref{theorem_1} to monotone drawings where also adjacent edges are allowed to cross.
Further, it would be interesting to generalize Theorem~\ref{theorem_1} to
arbitrary monotone drawings, where adjacent edges are also allowed to cross oddly. For such drawings, two notions of the crossing number are of interest. The {\em monotone odd crossing number}, $\text{mon-ocr}(G)$, counting the minimum number of pairs of edges crossing an odd number of times, and the {\em monotone independent odd crossing number}, $\text{mon-iocr}(G)$, or, $\text{mon-ocr}_-(K_n)$, counting the number of pairs of nonadjacent edges crossing an odd number of times. For every graph, we have $\text{mon-ocr}_-(G) \le \text{mon-ocr}(G) \le \text{mon-ocr}_{\pm}(G)$.

\subsection{Order types and \texorpdfstring{$\lambda$}{lambda}-matrices}

By Lemma~\ref{lemma2}, the crossing number of a semisimple drawing of $K_n$ is determined by the number of $k$-edges for all $k$. For a set of points $p_1,p_2,\dots,p_n$ in the plane, Goodman and Pollack\mycite{GP83_multidimensional} introduced the {\em $\lambda$-matrix\/} $(\lambda(i,j))$, where for every $i\neq j$, $\lambda(i,j)$ is the number of points to the left of the directed line $p_ip_j$, and $\lambda(i,i)=0$. They showed that the $\lambda$-matrix determines the order type of the point set. Aichholzer {\em et al.}\mycite{AK07_lambda} used $\lambda$-matrices to represent point sets for computing lower bounds on the rectilinear crossing number of complete graphs.

The $\lambda$-matrix may be defined for semisimple drawings of $K_n$ with vertices $v_1, v_2, \dots, v_n$ in a similar way: for every $i\neq j$, $\lambda(i,j)$ is the number of triangles $v_iv_jv_l$ oriented counter-clockwise. Clearly, $v_iv_j$ is a $k$-edge if and only if $\lambda(i,j)\in \{k,n-2-k\}$. The order type of a drawing determines its $\lambda$-matrix, but not the drawing itself (see Figure~\ref{fig_shellable_monotone_K4} or Figure~\ref{fig_shellable_monotone_K5}). Therefore, the $\lambda$-matrix does not determine the drawing either. However, a generalization of Goodman's and Pollack's result to semisimple drawings is true.

\begin{observation}
The $\lambda$-matrix of a semisimple drawing of $K_n$ determines its order type.
\end{observation}

This is easily seen by induction over all subgraphs of $K_n$: in every semisimple drawing of a graph with at least one edge, all edges incident with the outer face are $0$-edges. In particular, there is an edge $v_iv_j$ such that $\lambda(i,j)=0$. Every $0$-edge determines the orientation of all incident triangles. Therefore, we may remove such an edge, update the $\lambda$-matrix and use induction for the smaller graph.

The same observation is no longer true for weakly semisimple drawings: in the drawing in Figure~\ref{fig_weakly_shell_K4_ocr2}, right, every edge is a $1$-edge. Therefore, its $\lambda$-matrix is identical with the $\lambda$-matrix of a mirror-symmetric drawing, but these two drawings have mutually inverse order types.

Since the crossing number of a semisimple drawing of a complete graph is determined by its $\lambda$-matrix, it might be interesting to investigate the properties of $\lambda$-matrices that can be realized by semisimple drawings of complete graphs.

\section*{Acknowledgments}

%{\footnotesize
We would like to thank Pavel Valtr for initializing the research which led to this problem and Marek Eli\'{a}\v{s} for developing visualization tools that were helpful during the research.
%}

%---------------------------- Bibliography -------------------------------

\end{document}